\documentclass[reqno,10pt]{amsart}

\usepackage{amssymb,amsmath,amsthm,amsfonts}
\usepackage{color}
\usepackage{hyperref}
\usepackage{graphics,graphicx,wrapfig,epsf}
\usepackage{enumerate}
\usepackage{a4wide}

\theoremstyle{plain}
\newtheorem{prop}{Proposition}[section]
\newtheorem{theorem}[prop] {Theorem}
\newtheorem{lemma}[prop]{Lemma}
\newtheorem{definition}[prop]{Definition}
\newtheorem{corollary}[prop]{Corollary}
\numberwithin{equation}{section}

\theoremstyle{definition}

\newtheorem{assumption}[prop]{Assumption}

\theoremstyle{remark}
\newtheorem{remark}{Remark}

\newcommand{\bbN}{\mathbb{N}} \newcommand{\N}{\mathbb{N}}
\newcommand{\bbR}{\mathbb{R}} \newcommand{\R}{\mathbb{R}}
 
\newcommand{\bbC}{\mathbb{C}} \newcommand{\C}{\mathbb{C}} 
 \renewcommand{\P}{\mathbb{P}}
 \newcommand{\E}{\mathbb{E}}

\newcommand{\dd}{\mathrm{d}} 
\newcommand{\eps}{\epsilon}

\renewcommand{\eps}{\varepsilon}
\renewcommand{\phi}{\varphi}
\renewcommand{\Re}{\mathrm{Re}\,}
\renewcommand{\Im}{\mathrm{Im}\,}

\DeclareMathOperator{\const}{const}

\newcommand{\e}[1]{\,{\rm e}^{#1}\,} 
\newcommand{\ii}{{\rm i}}

\newcommand{\be}{\begin{equation}}
\newcommand{\ee}{\end{equation}}

\definecolor{darkgreen}{rgb}{0,0.6,0}

\begin{document}

\title{Singularity analysis for heavy-tailed random variables}

\author{Nicholas M.\ Ercolani, Sabine Jansen, Daniel Ueltschi}

\thanks{Department of Mathematics, The University of Arizona, Tucson, AZ 85721--0089, ({\tt ercolani@math.arizona.edu}). Supported by NSF grant DMS-1212167.}

\thanks{Mathematisches Institut, Ludwigs-Maximilians Universit{\''a}t M{\''u}nchen, Theresienstr. 39, 80333 Munich, Germany,
 ({\tt jansen@math.lmu.de}).}

\thanks{Department of Mathematics, University of Warwick, Coventry, CV4 7AL, United Kingdom ({\tt daniel@ueltschi.org}).}

\date{21 April 2018}



\maketitle

\begin{abstract}
	We propose a novel complex-analytic method for sums of i.i.d. random variables that are heavy-tailed and integer-valued. The method combines singularity analysis, Lindel{\"o}f integrals, and bivariate saddle points. As an application, we prove three theorems on precise large and moderate deviations which provide a local variant of a result by S.~V.~Nagaev (1973). The theorems generalize five theorems by A.~V.~Nagaev (1968) on stretched exponential laws $p(k) = c\exp( -k^\alpha)$ and apply to logarithmic hazard functions $c\exp( - (\log k)^\beta)$, $\beta>2$; they cover the big jump domain as well as the small steps domain. The analytic proof is complemented by clear probabilistic heuristics. Critical sequences are determined with a non-convex variational problem. 
\end{abstract}

\section{Introduction}

The motivation of the present article is two-fold. First, we present a new analytic method for the investigation of large powers of generating functions of sequences that satisfy some analyticity and log-convexity conditions. The method is explained and developed for probability generating functions but it has potentially broader applications and is motivated by techniques commonly used in analytic combinatorics~\cite{FS}. Specifically, we show that methods akin to singularity analysis can be pushed beyond the realm of functions amenable to singularity analysis in the sense of~\cite[Chapter VI.1]{FS}.

Second, we explore consequences for probabilistic limit laws and prove three theorems on precise large and moderate deviations for sums of independent identically distributed (i.i.d.) random variables that are heavy-tailed~\cite{extremalbook} and integer-valued. The theorems 
generalize results on stretched exponential laws by A.~V.~Nagaev~\cite{nagaev} which have recently attracted interest in the context of the zero-range process~\cite{agl}. They are close in spirit to results by S.~V.~Nagaev~\cite{nagaev73}, however with more concrete conditions on the domain of validity of the theorems, and 
provide deviations results ``on the whole axis''~\cite{rozovskii}. Our assumptions are more restrictive than one may wish from a probabilistic perspective; in return, they allow for sharp results and may provide a helpful class of explicit reference examples. For example, we prove that one of the bounds of the (local) big-jump domain for logarithmic hazard functions derived in~\cite{denisov-dieker-shneer} is sharp.

The analytic proof of the theorems is complemented by clear probabilistic heuristics. Our results cover different regimes: a small steps or moderate deviations regime, where a classical variant of a local central limit theorem with corrections expressed with the Cram{\'e}r series holds~\cite{ibragimov-linnik}, and a big-jump regime where the large deviation is realized by making one out of the $n$ variables large. In the language of statistical mechanics and the zero-range process, they correspond to supersaturated gas and a condensed phase~\cite{agl}. The critical scales that distinguish between regimes are defined with the help of a non-convex variational problem which encodes competing probabilistic effects. The variational problem has been analyzed before~\cite{nagaev,rozovskii}, our strong assumptions on the probability weights allow for a detailed analysis. Our results are further facilitated by the non-negativity of the random variables, which dispenses us from dealing with left tails.

The study of combinatorial generating functions shares much in common with the study of probability generating functions; in fact in many instances they coincide or run parallel as is the case for more recent investigations in the area of random combinatorial structures \cite{arratia-barbour-tavare}. From the viewpoint of complex function theory the key here involves relating asymptotic questions to questions about the nature of the singularities of generating functions viewed as more global analytic objects. In most of the successful applications of this {\it singularity analysis} to coefficient asymptotics a bridge is provided by the realization that the series in question satisfies some global algebraic or differential equation. Generating functions for which this is the case are referred to as {\it holonomic}. Pushing beyond this class in a systematic way requires new ideas and one of the most promising of these is the use of Lindel\"of integrals. Lindel\"of introduced these classically \cite{Lind} as a means to constructively carry out analytic continuations of function elements (series) in a fairly general setting. In more recent times his construction has begun to be used to study non-holonomic combinatorial generating functions \cite{FGS}. The generating functions for heavy-tailed distributions studied in this paper are of non-holonomic type and our methods of studying them show a new application of Lindel\"of's construction that has novel connections to other areas of analytic asymptotic analysis such as bivariate steepest descent. In future work we hope to build on the present article in a way that broadens the application of harmonic analysis and complex function theory to problems of asymptotic analysis in both probability and combinatorics, such as applying the theory of Hardy spaces and Riemann-Hilbert analysis and extensions of Tauberian theorems as originally envisioned by Paley and Wiener \cite{Feller}.

Our proof shares some features with~\cite{nagaev73}, where  cumulative distribution functions are approximate Laplace transforms and approximating moment generating function admit analytic extensions. Contour integrals that appear in inversion formulas are deformed and analyzed by Gaussian approximation---our proof details in Section~\ref{sec:gaussian} follow~\cite{nagaev73}. There are, however, key differences: we need not deal with approximation errors because of stronger analyticity assumptions, and our detailed analysis of the underlying variational problem allows us to formulate more concrete conditions for our theorems. 

The remainder of this article is organized as follows. In Section~\ref{sec results} we formulate our main results and discuss applications to stretched exponential weights $c\exp(-k^\alpha))$ and weights $c \exp( - (\log x)^\beta)$ with logarithmic hazard functions. In Section~\ref{sec:strategy} we explain the proof strategy in five steps, which are treated in detail in the remaining sections. Steps~1 and~2 concern analytic extensions and notably use the Lindel{\"o}f and Bromwich integrals (Section~\ref{sec:lindeloef}). Steps~4 and~5 analyze the critical points of a bivariate function and deal with the Gaussian approximation to a double integral (Section~\ref{sec:saddlepoint}). The pivotal Step~3 connects the contour integral and the bivariate double integral; it leads to the full proof of our theorems that can be found in Section~\ref{sec:asyana}. 

\section{Results} \label{sec results}

We use the notation $a_n\sim b_n$ if $a_n=(1+o(1))b_n$ and $a_n\ll b_n$ if $a_n = o(b_n)$. 

\subsection{Preliminaries} 

In order to formulate the results, we need to introduce critical sequences deduced from a variational problem and the Cram{\'e}r series. 
Let $X,X_1,X_2,\ldots$ be independent, identically distributed random variables with values in $\N$ and law 
\be
    \P(X=k) = p(k) = \exp(-q(k))\quad (k\in \N)
\ee 
for some sequence $(q(k))_{k\in \N}$. We assume that $X$ is heavy-tailed and has moments of all orders,  i.e., the generating function 
\be
  G(z) = \sum_{k=1}^\infty p(k) z^k \quad (|z| \leq 1)
\ee
has radius of convergence $1$ and  $\E[X^m] = \sum_{k=1}^\infty k^m p(k)< \infty$ for all $m\in\N$. Let $\mu$ and $\sigma^2$ be the expectation and variance of $X$. Set $S_n = X_1+\cdots + X_n$.
We are interested in the asymptotic behavior of $\P(S_n =  \mu n  + N_n)$ when $n,N_n\to \infty$ with $N_n \gg \sqrt{n}$. The following assumption is similar to conditions considered by S.~V.~Nagaev~\cite{nagaev73}. 

\begin{assumption} \label{ass:convexity}
	For some $a>0$, the sequence $(q(k))_{k\in \N\cap (a,\infty)}$ extends to a smooth function $q:(a,\infty) \to \R$ which has the following properties: 
	\begin{enumerate} 
		\item [(i)] $q'> 0$, $q''<0$, and $q'''>0$.
		\item [(ii)] $\lim_{x\to \infty} x q'(x)/(\log x) = \infty$.
		\item [(iii)] $c_1 \frac{q'(x)}{x} \leq |q''(x)| \leq c_2 \frac{q'(x)}{x}$ for some constants $c_1,c_2>0$. 
		\item [(iv)] $c_3 \frac{|q''(x)|}{x} \leq q'''(x) \leq c_4 \frac{|q''(x)|}{x}$ for some constants $c_3, c_4>0$. 
		\item[(v)] $q'(x)\leq \alpha \frac{q(x)}{x}$ for some $\alpha \in (0,1)$.
	\end{enumerate} 
\end{assumption}

Assumption~\ref{ass:convexity} allows for an easy analysis of an auxiliary variational problem, which is essential to the formulation of our main results. 
Let us collect a few elementary consequences. Under Assumption~\ref{ass:convexity}, $q$ is concave on $(a,\infty)$ and $p = \exp(-q)$ is log-convex. Moreover, $\lim_{x\to \infty} x^2 q''(x)/\log x = - \infty$ and for $y >x >a$, 
using 
\be 
	\frac{q'(y)}{q'(x)} = \exp\Bigl( - \int_x^y \frac{|q''(u)|}{q'(u)} \dd u\Bigr),
\ee 
we estimate 
\be \label{eq:qprimevar}
	\Bigl( \frac{y}{x}\Bigr)^{ - c_2} \leq \frac{q'(y)}{q'(x)} \leq \Bigl( \frac{y}{x}\Bigr)^{ - c_1} \leq 1
\ee 
Similarly, for $y>x>a$, 
\be  \label{eq:qsecvar}
	\Bigl( \frac{y}{x}\Bigr)^{ c_3} \leq \frac{q''(y)}{q''(x)} \leq \Bigl( \frac{y}{x}\Bigr)^{ c_4}. 
\ee
Since $G(z) =\sum_k z^k \exp(- q(k))$ has radius of convergence $1$, we also know that
\be \label{eq:qprimetozero}
	\lim_{x\to \infty} q'(x)= \lim_{x\to \infty} q''(x)  = \lim_{x\to \infty} q'''(x)= 0. 
\ee
Indeed by Assumption~\ref{ass:convexity}, $q'$ is eventually decreasing and the limit $\ell:= \lim_{x\infty} q'(x)$ exists in $\R \cup \{-\infty\}$. Then $\ell = \lim_{x\to \infty} q(x)/x$ and $G(z)$ has radius of convergence $\exp(\ell) =1$, whence $\ell =0$. Assumption~\ref{ass:convexity}(iii) and~(iv) leads to the statements on higher order derivatives. Assumption~\ref{ass:convexity}(v) implies $q(x)= O(x^\alpha)$ as $x\to \infty$.

Our method of proof requires two more analyticity assumptions.

\begin{assumption} \label{ass:analyticity}
	There exists $b\geq 0$ such that $(p(n))_{n\in \N\cap [b,\infty)}$ 
 extends to a function $p(\zeta)$ that is 
	continuous on a closed half-plane $\Re \zeta \geq b$, analytic on the open half-plane $\Re \zeta> b$, and in addition satisfies 
	\begin{enumerate}
		\item [(i)] For every $\eps\in (0,\pi)$, some $C_\eps>0$, and all $\zeta$, we have $|p(\zeta)| \leq C_\eps \exp(\eps |\zeta|)$. 
		\item [(ii)] $\int_{-\infty}^\infty |(b+\mathrm{i}s)^k p(b+ \mathrm{i} s)|\dd s <\infty$ for all $k\in \N$.
	\end{enumerate}
	Moreover $p(x) = \exp(-q(x))$ for all $x\geq \max(a,b)$ with $a$, $q(x)$ as in Assumption~\ref{ass:convexity}. 
\end{assumption} 

\begin{assumption} \label{ass:analyticity2}
	Let $p(\zeta) = \exp( - q(\zeta))$ be the analytic extension from Assumption~\ref{ass:analyticity}, defined in $\Re \zeta \geq b$. Then $q(\zeta) = - \mathrm{Log} \zeta$, defined with the principal branch of the logarithm is analytic as well, and the following holds: 
	\begin{enumerate}
		\item [(i)] For $r>0$ large, let $z_r = b + \ii \sqrt{r^2 - b^2}$. Then as $r\to \infty$,
		$$\left|\int_{\Re \zeta = b,\, |\zeta| \geq r}	\exp(- \Re q(\zeta)) \dd \zeta \right| \leq \exp( - \Re q(z_r ) + O(\log r)).$$
		\item [(ii)] $\Im (\zeta q'(\zeta)) \leq \Im (\zeta q'(r))$ for all large $r$ and all  $\zeta$ with $\Im \zeta \geq 0$ and $|\zeta|=r$. 
		\item [(iii)] $|q'''(\zeta)|\leq C |q''(\zeta)/\zeta|$ for some $C>0$ and all $\zeta$. 
	\end{enumerate}
\end{assumption}
\noindent Assumption~\ref{ass:analyticity2} enters the proof of Theorem~\ref{thm:imsmall} only. 


\subsubsection*{Variational problem and critical scale} 
Assumption~\ref{ass:convexity} is tailored to the analysis of an auxiliary variational problem (see also~\cite{nagaev, rozovskii}), motivated by the following heuristics.  
For subexponential random variables, the typical large deviations behavior is realized by making one out of the $n$ variables large,
\be \label{eq:heuristics1}
	\P(S_n = \mu n + N_n )\approx n \P(X_n=N_n - k_n) \P(S_{n-1} = \mu n + k_n )
\ee
with a yet to be determined optimal $k_n$. Assuming that a normal approximation for the second factor is justified, we get
\be \label{eq:heuristics2}
	\P(S_n = \mu n + N_n)\approx \exp\Bigl( - q(N_n -k_n) - \frac{k_n^2}{2 n\sigma^2}\Bigr)
\ee
where we have neglected prefactors $n$ and $1/\sqrt{2 \pi n \sigma^ 2}$ (see Eq.~\eqref{eq:refined-heuristics} below for a more refined heuristics). The optimal $k_n$ is then determined by minimizing the term in the exponential. Thus we are led to the minimization of 
\be
	 f_n(x) = q(x) + \frac{(N_n-x)^2}{2n \sigma^2}. 
\ee
As illustrated in Figure~\ref{fig variational}, the nature of the variational problem changes with $N_n$. Define $x_n^*>0$ and $N_n^*$ by 
\be \label{eq:xnstar}
	q''(x_n^*) = - \frac{1}{n\sigma^2},\quad N_n^* = x_n^* + n\sigma^2 q'(x_n^*). 
\ee
For sufficiently large $n$, the inflection point $x_n^*$ is uniquely defined because of the monotonicity from Assumption~\ref{ass:convexity} and Eq.~\eqref{eq:qprimetozero}, moreover $x_n^*\to \infty$. The quantity $N_n^*$ is defined in such a way that the tangent to the curve $y = q\rq{}(x)$ at $x=x_n^*$ has equation $y = (N_n^*- x)/(n\sigma^2)$, see Fig.~\ref{fig critical}. 

The next two lemmas characterize the minimization of $f_n$; they are proven in Section~\ref{sec:variational}. The first lemma relates the critical points of $f_n$ to the location of $N_n$ compared to $N_n^*$. 
  
\begin{lemma} \label{lem:avar}
	For sufficiently large $n$, the following holds true:
	\begin{enumerate} 
		\item [(a)] If $N_n < N_n^*$, then $f_n'>0$ on $(a,\infty)$. 
		\item [(b)] If $N_n = N_n^*$, then $f'_n$ has the unique zero $x_n^*$, moreover $f'_n(x)\geq 0$ with equality if and only if $x =  x_n^*$. 
		\item [(c)] If $N_n > N_n^*$ and $\limsup_{n\to\infty}N_n/(n \sigma^2) < \lim_{x\searrow a} q'(x)$, then $f_n$ has exactly two critical points $x_n$ and $x'_n$, which satisfy $x'_n < x_n^*< x_n<N_n$ and 
	$$ f_n(x'_n)= \max_{(a,x_n^*)}f_n, \quad f_n(x_n)= \min_{(x_n^*,\infty)} f_n.$$
		\item [(d)] If $N_n >N_n^*$ and $\liminf_{n\to\infty}N_n/(n \sigma^2) > \lim_{x\searrow a} q'(x)$, then $f_n$ has a unique critical point $x_n$. It satisfies $x_n\in (x_n^*, N_n)$ and is a global minimizer. 
		\end{enumerate} 
\end{lemma} 
\noindent 
For $N_n>N_n^*$, the function $f_n$ may have two local minimizers: $a$ and $x_n$, and we may wonder which one is the global minimizer. The answer depends on the location of $N_n$ compared to a new critical sequence $N_n^{**}$. Concrete examples are given in Sections~\ref{sec:stretched} and~\ref{sec:logarithmic}. 

\begin{lemma} \label{lem:breakeven}
	For $n$ sufficiently large, there is a uniquely defined $N_n^{**} >N_n^*$ such that:
	\begin{enumerate} 
		\item [(a)] If $N_n^*< N_n < N_n^{**}$, then $f_n(a)<f_n(x_n)$. 
		\item [(b)] If $N_n = N_n^{**}$, then $f_n(a)=f_n(x_n)$. 
		\item [(c)] If $N_n>N_n^{**}$, then $f_n(x_n)< f_n(a)$. 
	\end{enumerate}
\end{lemma} 

\noindent In general it may not be straightforward to determine $N_n^{**}$ exactly, but it is simple to find a lower bound: if $a=0$ and $q(N_n) < N_n^2/(2n \sigma^2)$, then $N_n >N_n^{**}$. This lower bound corresponds, roughly, to the sequence $\Lambda(n)$ in~\cite{nagaev73}. 

 The sequences introduced up to now are ordered as follows. 

\begin{lemma} \label{lem:critsequences}
	As $n\to \infty$, we have 
	$$\sqrt{n} \ll x_n^* <N_n^* <N_n^{**} =O(n^{1/(2-\alpha)})\ll n,$$
	and for some constants $C,\delta >0$
	$$ (1+\delta )x_n^* \leq N_n ^* \leq C x_n ^*. $$
\end{lemma}

\noindent The lemma is proven in Section~\ref{sec:variational}. Lemmas~\ref{lem:avar} and~\ref{lem:breakeven}, together with the heuristics described above, suggest that for $N_n > N_n^{**}$ the unlikely event $S_n = n \mu + N_n$ is realized by making one component of the order of $x_n$. One may wonder how far $x_n$ is from swallowing all of the overshoot $N_n$. 

\begin{centering}
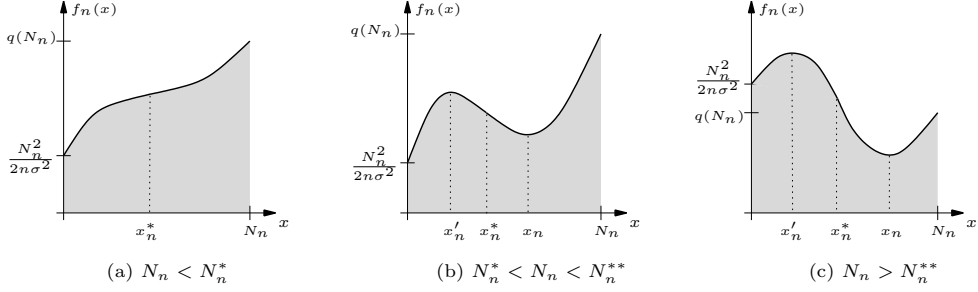
\begin{figure}[htb]
\begin{picture}(0,0)%
\includegraphics{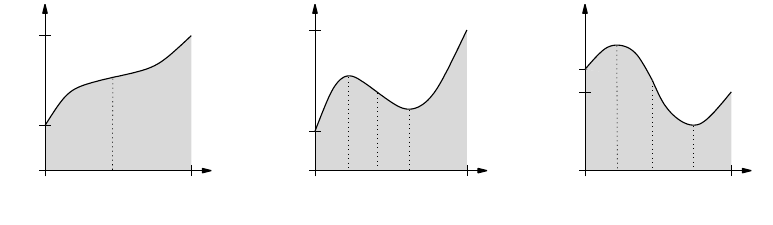}%
\end{picture}%
\setlength{\unitlength}{2368sp}%
\begingroup\makeatletter\ifx\SetFigFont\undefined%
\gdef\SetFigFont#1#2#3#4#5{%
  \reset@font\fontsize{#1}{#2pt}%
  \fontfamily{#3}\fontseries{#4}\fontshape{#5}%
  \selectfont}%
\fi\endgroup%
\begin{picture}(10296,2986)(1,-4685)
\put(1351,-4186){\makebox(0,0)[lb]{\smash{{\SetFigFont{5}{6.0}{\rmdefault}{\mddefault}{\updefault}{\color[rgb]{0,0,0}$x_n^*$}%
}}}}
\put(4276,-1861){\makebox(0,0)[lb]{\smash{{\SetFigFont{5}{6.0}{\rmdefault}{\mddefault}{\updefault}{\color[rgb]{0,0,0}$f_n(x)$}%
}}}}
\put(7876,-1861){\makebox(0,0)[lb]{\smash{{\SetFigFont{5}{6.0}{\rmdefault}{\mddefault}{\updefault}{\color[rgb]{0,0,0}$f_n(x)$}%
}}}}
\put(6526,-4111){\makebox(0,0)[lb]{\smash{{\SetFigFont{6}{7.2}{\rmdefault}{\mddefault}{\updefault}{\color[rgb]{0,0,0}$x$}%
}}}}
\put(2851,-4111){\makebox(0,0)[lb]{\smash{{\SetFigFont{6}{7.2}{\rmdefault}{\mddefault}{\updefault}{\color[rgb]{0,0,0}$x$}%
}}}}
\put(10051,-4111){\makebox(0,0)[lb]{\smash{{\SetFigFont{6}{7.2}{\rmdefault}{\mddefault}{\updefault}{\color[rgb]{0,0,0}$x$}%
}}}}
\put(  1,-3436){\makebox(0,0)[lb]{\smash{{\SetFigFont{5}{6.0}{\rmdefault}{\mddefault}{\updefault}{\color[rgb]{0,0,0}$\frac{N_n^2}{2n\sigma^2}$}%
}}}}
\put(2476,-4186){\makebox(0,0)[lb]{\smash{{\SetFigFont{5}{6.0}{\rmdefault}{\mddefault}{\updefault}{\color[rgb]{0,0,0}$N_n$}%
}}}}
\put(6151,-4186){\makebox(0,0)[lb]{\smash{{\SetFigFont{5}{6.0}{\rmdefault}{\mddefault}{\updefault}{\color[rgb]{0,0,0}$N_n$}%
}}}}
\put(9676,-4186){\makebox(0,0)[lb]{\smash{{\SetFigFont{5}{6.0}{\rmdefault}{\mddefault}{\updefault}{\color[rgb]{0,0,0}$N_n$}%
}}}}
\put(4576,-4186){\makebox(0,0)[lb]{\smash{{\SetFigFont{5}{6.0}{\rmdefault}{\mddefault}{\updefault}{\color[rgb]{0,0,0}$x_n'$}%
}}}}
\put(4951,-4186){\makebox(0,0)[lb]{\smash{{\SetFigFont{5}{6.0}{\rmdefault}{\mddefault}{\updefault}{\color[rgb]{0,0,0}$x_n^*$}%
}}}}
\put(5401,-4186){\makebox(0,0)[lb]{\smash{{\SetFigFont{5}{6.0}{\rmdefault}{\mddefault}{\updefault}{\color[rgb]{0,0,0}$x_n$}%
}}}}
\put(3601,-3511){\makebox(0,0)[lb]{\smash{{\SetFigFont{5}{6.0}{\rmdefault}{\mddefault}{\updefault}{\color[rgb]{0,0,0}$\frac{N_n^2}{2n\sigma^2}$}%
}}}}
\put(7201,-2611){\makebox(0,0)[lb]{\smash{{\SetFigFont{5}{6.0}{\rmdefault}{\mddefault}{\updefault}{\color[rgb]{0,0,0}$\frac{N_n^2}{2n\sigma^2}$}%
}}}}
\put(8101,-4186){\makebox(0,0)[lb]{\smash{{\SetFigFont{5}{6.0}{\rmdefault}{\mddefault}{\updefault}{\color[rgb]{0,0,0}$x_n'$}%
}}}}
\put(8626,-4186){\makebox(0,0)[lb]{\smash{{\SetFigFont{5}{6.0}{\rmdefault}{\mddefault}{\updefault}{\color[rgb]{0,0,0}$x_n^*$}%
}}}}
\put(9151,-4186){\makebox(0,0)[lb]{\smash{{\SetFigFont{5}{6.0}{\rmdefault}{\mddefault}{\updefault}{\color[rgb]{0,0,0}$x_n$}%
}}}}
\put(7201,-2986){\makebox(0,0)[lb]{\smash{{\SetFigFont{5}{6.0}{\rmdefault}{\mddefault}{\updefault}{\color[rgb]{0,0,0}$q(N_n)$}%
}}}}
\put(3601,-2086){\makebox(0,0)[lb]{\smash{{\SetFigFont{5}{6.0}{\rmdefault}{\mddefault}{\updefault}{\color[rgb]{0,0,0}$q(N_n)$}%
}}}}
\put(8401,-4636){\makebox(0,0)[lb]{\smash{{\SetFigFont{7}{8.4}{\rmdefault}{\mddefault}{\updefault}{\color[rgb]{0,0,0}(c)  $N_n > N_n^{**}$}%
}}}}
\put(4501,-4636){\makebox(0,0)[lb]{\smash{{\SetFigFont{7}{8.4}{\rmdefault}{\mddefault}{\updefault}{\color[rgb]{0,0,0}(b)  $N_n^* < N_n < N_n^{**}$}%
}}}}
\put(1051,-4636){\makebox(0,0)[lb]{\smash{{\SetFigFont{7}{8.4}{\rmdefault}{\mddefault}{\updefault}{\color[rgb]{0,0,0}(a)  $N_n < N_n^*$}%
}}}}
\put(  1,-2161){\makebox(0,0)[lb]{\smash{{\SetFigFont{5}{6.0}{\rmdefault}{\mddefault}{\updefault}{\color[rgb]{0,0,0}$q(N_n)$}%
}}}}
\put(676,-1861){\makebox(0,0)[lb]{\smash{{\SetFigFont{5}{6.0}{\rmdefault}{\mddefault}{\updefault}{\color[rgb]{0,0,0}$f_n(x)$}%
}}}}
\end{picture}%
\caption{Minimization of $f_n(x) = q(x)+ (N_n - x)^2/(2 n\sigma^2)$ and illustration of Lemmas~\ref{lem:avar} and~\ref{lem:breakeven} for weights $q:(0,\infty)\to \R$ with $q(0)=0$. 
For $N_n>N_n^{*}$, $f_n(x)$ has two critical points $x'_n$ and $x_n$ separated by an inflection point $x_n^*$. The global minimum is reached either at $x=x_n$ or at $x=0$.}
\label{fig variational}
\end{figure}
\end{centering}

\begin{lemma} \label{lem:crit-magnitudes}
	Suppose Assumption~\ref{ass:convexity}(i) holds true. 
	Let $N_n >N_n^*$. Then 
	\begin{equation*}
	N_n - N_n^*\leq x_n\leq N_n,\quad 
	n\sigma^2 f''_n(x_n) = 1 - n\sigma^2 |q''(x_n) | = 1+O\Bigl(\frac{N_n^*}{N_n}\Bigr). 
	\end{equation*}
\end{lemma} 
\noindent In particular, for $N_n\gg N_n^*$, we have $x_n \sim N_n$ and $n\sigma^2 f''(x_n) \to 1$. The lemma is proven in Section~\ref{sec:variational}. The information on the second derivative enters a refined heuristics: we  make the ansatz that conditional on the unlikely event $S_n = n\mu +N_n$, there is one large component of size $x_n$, but the size is not deterministic. Instead, there are fluctuations around $x_n$. This yields 
\begin{align}
	\P(S_n = n \mu + N_n) & \approx n \sum_{\ell} \P(X_1 = x_n  + \ell) \P(S_{n-1} = N_n - x_n - \ell) \notag \\
		& \approx n \sum_\ell \frac{\exp(- f_n(x_n + \ell))}{\sqrt{2\pi n \sigma^2}}  
	 \approx n  \e{- f_n(x_n)} \sum_\ell \frac{\exp(- f''_n(x_n) \ell^2 /2)}{\sqrt{2\pi n \sigma^2}}\notag \\
		 &\approx
		\frac{n\exp( - f_n(x_n))}{\sqrt{1- n \sigma^2 |q''(x_n)|}}. \label{eq:refined-heuristics}
\end{align}
Theorem~\ref{thm:large} below confirms the heuristics for large $N_n$, up to correction terms both in the prefactor and in the exponential.

\subsubsection*{Cumulants and Cram{\'e}r series}

The heuristics together with Lemma~\ref{lem:crit-magnitudes} suggest that the optimal $k_n=N_n - x_n$ in Eq.~\eqref{eq:heuristics2} is of order up to $N_n^*\gg \sqrt{n}$. At this scale the normal approximation fails and requires  correction terms. The latter are usually expressed with  the Cram{\'e}r series~\cite{ibragimov-linnik}, whose definition we briefly recall. Let $\phi(t)$ be the cumulant generating function of $X$, 
\be 
   \phi(t) =\log  \E[\e{tX}] = \log G(\e{t}) \quad (\Re t\leq 0).
\ee
Note $\phi(0)=0$. As $t\to 0$, $\phi(t)$ can be approximated to arbitrary order 
\be \label{eq:phitaylor}
   \phi(t) = \sum_{j=1}^r \kappa_j \frac{t^j }{j!} + O(t^{r+1})
\ee
with finite and real expansion coefficients $\kappa_j \in \bbR$, the cumulants, see Section~\ref{sec:lindeloef}.
Notice that $\kappa_1 = \mu$ is the expectation and $\kappa_2 =\sigma^2>0$ is the variance of $X$. 

\begin{definition} \label{def:cramer}
 Let $t(\tau) =\frac{1}{\sigma^2} \tau+  \sum_{j\geq 2} a_j \tau^j$ be the formal power series obtained by inverting 
 \begin{equation*} 
     \sigma^2 t(\tau) + \sum_{j \geq 2} \kappa_{j+1} \frac{t(\tau)^j}{j!} =  \tau. 
  \end{equation*} 
 The \emph{Cram{\'e}r series} $\sum_{j\geq 0} \lambda_j \tau^j$ is the formal power series defined by composing the expansion of $t(\tau)$ with the expansion of $(\mu+\tau) t - \phi(t)$, 
 \begin{equation*} 
  (\mu+ \tau) t(\tau) - \sum_{j\geq 1} \kappa_j \frac{t(\tau)^j}{j!} = -\frac{\tau^2}{2\sigma^2} + \tau^3 \sum_{j\geq 0} \lambda_j \tau^j. 
 \end{equation*} 
\end{definition}

\noindent  Equivalently, the Cram{\'e}r series is the left-sided Taylor expansion of the Legendre transform $\phi^*$ at $\mu$: 
 let $\phi^*(x):= \sup_{t\leq 0} (t x - \phi(t))$. Then as $\tau \nearrow 0$, 
 \be \label{eq:craleg}
   \phi^*(\mu + \tau) = - \frac{\tau^2}{2\sigma^2} +\tau^3 \sum_{j= 0}^r \lambda_j \tau^{j} +  O(\tau^{r+4})
 \ee
 to arbitrarily high order $r$. 
 
 \begin{remark} 
 For $t>0$, $\log [\sum_{k \geq 1} p(k) \exp( k t)]$ is infinite and the standard convention is to set $\phi(t) = \infty$; then $\phi^*(\mu+ \tau) \equiv 0$ for $\tau \geq 0$ and Eq.~\eqref{eq:craleg} no longer applies. We adopt a different convention, however, for which $\varphi(t)$ is smooth in a neighborhood of $0$ (see Theorem~\ref{thm:bval}), though it becomes complex-valued, and Eq.~\eqref{eq:craleg} applies to $(\Re \varphi)^*$ for positive $\tau$ as well. 
\end{remark}

\subsection{Main theorems} 
Set $f_{n0} = f_n$ and for $r\geq 1$, 
\be
	f_{nr}(x) = q(x) + \frac{(N_n-x)^2}{2 n \sigma^2}- \frac{(N_n - x)^3}{n^2} \sum_{j=0}^{r-1} \lambda_j \Bigl(\frac{N_n - x}{n}\Bigr)^{j}.
\ee
Remember the minimization of $f_n(x)$ and the critical scales $\sqrt{n}\ll N_n^* < N_n^{**} =O(n^{1/(2-\alpha)}) \ll n$ illustrated in Fig.~\ref{fig variational}. In Proposition~\ref{prop:truncated} below we check that the properties of $f_n$ carry over to $f_{nr}$. 
The following theorems provide a local variant of a large deviations theorem by S.~V.~Nagaev~\cite{nagaev73}, see also~\cite[Theorem 2.1]{nagaev79}. 
Local results have been provided before, see~\cite{denisov-dieker-shneer} and the references therein. 
The principal difference, apart from the local character of the theorems, is that our detailed investigation of the variational problem and notably Lemma~\ref{lem:crit-magnitudes} allows us to formulate conditions directly in terms of $N_n$, whereas S.~V.~Nagaev's criteria included an indirect condition on the sign of some second derivative. 

\begin{theorem} \label{thm:moderate}
	Let $N_n\to \infty$ with $\sqrt{n}\ll N_n \leq (1+o(1))N_n^*$. Pick $r$ large enough so that $n (N_n/n)^{r}\to 0$. Then	
	$$
		\P(S_n = \mu n + N_n) \sim \frac{1}{\sqrt{2\pi\sigma^2 n}} \exp\Bigl( - \frac{N_n^2}{2n \sigma^2} + \frac{N_n^3}{n^2}\sum_{j=0}^{r-1} \lambda_j \Bigl(\frac{N_n}{n}\Bigr)^j \Bigr).
	$$
\end{theorem}

\begin{theorem} \label{thm:critical} 
	Let $N_n\to \infty$ with $\liminf N_n / N_n^* >1$ and $N_n = O(n^{1/[2-\alpha]} )$. Pick $r$ large enough so that $n (N^{*}_n/n)^{r}\to 0$. Then 
	\begin{multline*}
		\P(S_n = \mu n  + N_n) =(1+o(1)) \frac{1}{\sqrt{2\pi\sigma^2 n}} \exp\Bigl( - \frac{N_n^2}{2n \sigma^2} + \frac{N_n^3}{n^2}\sum_{j=0}^{r-1} \lambda_j \Bigl(\frac{N_n}{n}\Bigr)^j \Bigr) \\
			+ (1+o(1)) \frac{n}{\sqrt{1- n \sigma^2 |q''(x_{nr})|}} \exp\Bigl( - f_{nr}(x_{nr})\Bigr)
	\end{multline*}
	with $x_{nr}= N_n + O(N_n^*)$ the largest solution of $f'_{nr}(x_{nr}) =0$. 
\end{theorem} 

\noindent 
Lemma~\ref{lem:breakeven} suggests that for $N_n \ll N_n^{**}$, the first contribution dominates and for $N_n \gg N_n^{**}$ the second contribution wins, but one has to be careful because of the factors $n$ and $1/\sqrt{2\pi n \sigma^2}$ as well as the Cram{\'e}r corrections; a detailed evaluation is best left to concrete examples (see, however, Corollary~\ref{cor:big-jump} below). 

\begin{theorem}\label{thm:large} 
	Let $N_n\to \infty$ with $N_n\gg n^{1/(2-\alpha)}$. Pick $r$ large enough so that $n (N^{*}_n/n)^{r}\to 0$. Then
	\begin{equation*}
		\P(S_n = \mu n + N_n) \sim n \exp\Bigl( - f_{nr}(x_{nr})\Bigr)
	\end{equation*}	 
	with $x_{nr}= N_n + O(N_n^*)$ the largest solution of $f'_{nr}(x_{nr}) =0$.
\end{theorem} 

\noindent 
In practice one may prefer not to deal with the Cram{\'e}r corrections or the variational problem, and the following proposition is helpful.

\begin{prop}\label{prop:truncated}
	Suppose that $\liminf_{n\to \infty} N_n/N_n^*>1$. Fix $r\in \N_0$. Then, for sufficiently large $n$, $x_{nr}$ is the maximizer of $f_{nr}$ restricted to $(x_n^*,N_n)$ and the unique zero of $f'_{nr}$ in that interval. Moreover $1- n \sigma^2 |q''(x_{nr})| = 1 + O(N_n^*/N_n)$ stays bounded away from zero and 
	\begin{align*}
		x_{nr} & = N_n - (1+o(1))n \sigma^2 q'(x_{nr}) = N_n  + O(N_n^*), \\
		f_{nr}(x_{nr}) & = q(x_{nr}) + \frac{1}{2}(1+o(1)) n \sigma^2 q'(x_{nr}) = q(N_n) \Bigl(1 + O\Bigl(\frac{N_n^*}{N_n}\Bigr)\Bigr). 
	\end{align*}
\end{prop}
\noindent The proposition is proven in Section~\ref{sec:variational}.  
%
For $N_n \gg N_n^*$, we obtain $x_{nr} \sim N_n$, $q'(x_{nr})\sim q'(N_n)$, and 
$q(x_{nr}) = q(N_n) - (1+o(1)) n \sigma^2 q'(N_n)$, hence 
\be \label{eq:fnrq}
	f_{nr}(x_{nr}) = q(N_n) - \frac{1}{2}(1+o(1)) n\sigma^2 q'(N_n)^2.
\ee
Now suppose in addition that $\liminf \frac{N_n^2}{2n \sigma^2}/q(N_n) >1$. Then in Theorem~\ref{thm:critical}, the first summand is of order $\exp( - (1+o(1)) N_n^2/(2n \sigma^2)$, the second of order $\exp( - (1+o(1))q(N_n))$, so the first contribution is negligible and the validity of Theorem~\ref{thm:large} extends accordingly, since $1- n \sigma^2 |q''(x_{nr})|\to 1$ for $N_n\gg N_n^*$. Eq.~\eqref{eq:fnrq} now yields the following corollary. 

\begin{corollary} \label{cor:big-jump} 
	Take $N_n\to \infty$ with $N_n \gg N_n^*$ and $\liminf \frac{N_n^2}{2n \sigma^2} /q(N_n)>1$. Then
	$$ \P(S_n = n\mu + N_n) \sim n \e{-q (N_n)} = n \P(X=N_n)$$
	if and only if $\sqrt{n\sigma^2} q'(N_n) \to 0$.
\end{corollary} 

\noindent In concrete examples, Theorem~\ref{thm:moderate} should allow us to extend the domain of validity of the corollary to $N_n \gg N_n^{**}$.  The condition $\sqrt{n\sigma^2} q'(N_n) \to 0$ is closely related to the \emph{insensitivity} scale discussed by Denisov, Dieker and Shneer~\cite{denisov-dieker-shneer}, as 
\be
	\frac{p(N_n \pm  \sqrt{n \sigma^2}\bigr) }{p(N_n)} \to 1 \ \Leftrightarrow  \sqrt{n\sigma^2} q'(N_n)\to 0. 
\ee

\begin{remark}	[Big-jump vs small steps]
	The domain where $\P(S_n = n \mu + N_n) \sim n \P(X=N_n)$ is sometimes called \emph{big-jump domain}. Think of $S_n$ as the position of a random walker with step size distribution $p(k)$. In the situation of Corollary~\ref{cor:big-jump}, the unlikely event that the walker has travelled a distance  $\mu n + N_n$ much larger than the expected distance $\mu n$ is realized by one big jump of size $N_n$. Finding the boundary of the big-jump domain is an active field of research~\cite{denisov-dieker-shneer}. 
	
	The interpretation of Theorem~\ref{thm:moderate}, in contrast,  is that the  moderate overshoot $N_n$ is achieved by a collective effort: all steps tend to stay small, though each stretches a little beyond its expected value $\mu$. For stretched exponential variables, this interpretation is made rigorous in~\cite{nagaev} and~\cite{agl}. 
\end{remark} 

\begin{remark}[Condensation in the zero-range process] 
 	In the zero-range process, the random variables $X_1,\ldots,X_n$ model the number of particles at lattice sites $j=1,\ldots,n$, with $S_n$ the total number of particles, and $\mu$ is a critical density. Theorem~\ref{thm:moderate} corresponds to supersaturated gas. In Corollary~\ref{cor:big-jump}, the particle excess $N_n$ is absorbed by a condensate, i.e., one large occupation number. In Theorem~\ref{thm:large}, the particle excess is shared by a condensate of size $x_{nr}<N_n$ and supersatured gas. See~\cite{agl} and~\cite[Section 7]{eju}.  	
\end{remark}

We conclude with an equivalent but more intrinsic formulation of Theorem~\ref{thm:large}. In Section~\ref{sec:lindeloef} we shall see that $G(z)$ extends to a function that is analytic in the slit plane $\C\setminus [1,\infty)$, and in addition the limit $G(\e{t}) = \lim_{\eps \searrow 0} G(\e{t} + \ii \eps)$ exists for all $t\geq 0$. So the cumulant generating function $\varphi(t) = \log G(\e{t})$ extends to a function that is well-defined and smooth in a neighborhood of the origin in $\R$, and Eq.~\eqref{eq:phitaylor} stays valid for small positive $t$. We define
\begin{align}
     \Phi_n(t,\zeta) & = - q(\zeta) + n \Re \phi(t) - (\mu n + N_n) t + t\zeta \notag  \\
         & = - q(\zeta) + n \sum_{j=2}^{r-1} \kappa_j \frac{t^j}{j!} - t \bigl(N_n - \zeta) + O(t^{r}).  \label{eq:sndef}
 \end{align} 
The asymptotic expansion holds for every order $r$. For $N_n\gg N_n^*$ and sufficiently large $n$, the bivariate function $\Phi_n(t,\zeta)$ has exactly two critical points in $(0, \frac{N_n}{n\sigma^2})\times (a,\infty)$. We label them as $(t_n,\zeta_n)$ and $(t'_n,\zeta'_n)$ with $t_n< t'_n$. Then, in the situation of Theorem~\ref{thm:large}, we have
\be \label{eq:bivariatehn}
	\P(S_n = n\mu + N_n) \sim  n\, \e{\Phi_n(t_n,\zeta_n)}. 
\ee
It is in this form that we prove the theorem. Let us explain how to recover the expression in terms of $f_{nr}$. We may solve for $\nabla \Phi_n(t,\zeta)=0$ in two steps: (1) use $\partial_ t \Phi_n(t,\zeta) = 0$ to express $t=t(\zeta)$ as a function of $\zeta$, (2) 
plug the expression into $\partial_\zeta \Phi_n (t,\zeta)$ to obtain an equation for $\zeta$. 
This latter step breaks into the following two stages:
\begin{itemize}
\item[(2a)] substitute the expression of $t= t(\zeta)$ into the expression for $\Phi_n(t,\zeta)$ so as to obtain a function $\Phi_n(t(\zeta),\zeta)$ of $\zeta$ only;
\item[(2b)] set the derivative of $\Phi_n(t(\zeta),\zeta)$ with respect to $\zeta$ to zero,
\end{itemize}
which is valid since
\be 
  \frac{\dd}{\dd \zeta} \Phi_n(t(\zeta),\zeta) = \frac{\partial \Phi_n}{\partial \zeta}(t(\zeta),\zeta) + \frac{\partial \Phi_n}{\partial t} ( t(\zeta),\zeta) \frac{\dd t}{\dd \zeta}(t) = \frac{\partial \Phi_n}{\partial \zeta}(t(\zeta),\zeta).
\ee
By the definition of the Cram{\'e}r series, step (2a) gives 
\be
	\Phi_n(t(\zeta),\zeta)   = - q(\zeta) - \frac{(N_n- \zeta)^2}{2n\sigma^2} +  \frac{(N_n- \zeta)^3}{n^2} \sum_{j\geq 0} \lambda_j \Bigl(\frac{N_n- \zeta}{n}\Bigr)^j
\ee
Truncation of the asymptotic expansion on the right-hand side gives precisely the function $- f_{nr}(\zeta)$. For $r$ large enough, step (2b) shows that we may approximate 
\be \label{eq:phitrunc}
	\Phi_n(t_n,\zeta_n) = - f_{nr}(x_{nr})+ o(1), \quad n\sigma^2 q''(\zeta_n) = n\sigma^2 q''(x_{nr})+o(1)
\ee	
hence the equivalence of Eq.~\eqref{eq:bivariatehn} with the expression from Theorem~\ref{thm:large}.

\subsection{Application to stretched exponential laws}\label{sec:stretched}

Here we explain how to recover five theorems by A.~V.~Nagaev~\cite{nagaev} for stretched exponential variables. Let $\alpha \in (0,1)$, $c>0$, and 
\be \label{eq:stretched} 
	p(k) = c \exp(-k^\alpha),\quad q(k) =  k^\alpha - \log c \quad (k \in \N). 
\ee
We need not check Assumption~\ref{ass:analyticity2} since Theorem~\ref{thm:imsmall} for stretched exponential weights has already been proven in~\cite[Theorem 2.4.6]{ibragimov-linnik}. 

\begin{lemma} \label{lem:stretched}
	The probability weights~\eqref{eq:stretched} satisfy Assumptions~\ref{ass:convexity} and~\ref{ass:analyticity}, and we have 
	$$ x_n^* = \bigl[ \alpha(1-\alpha) n \sigma^2\bigr]^{1/(2-\alpha)},\quad 
	N_n^* = \frac{2-\alpha}{1-\alpha}x_n^*,\quad N_n^{**} = C_\alpha (n\sigma^2)^{1/(2-\alpha)}
	$$	
	with $C_\alpha = (2-\alpha)(2-2\alpha)^{- (1-\alpha)/(2-\alpha)}$. Moreover $\sqrt{n\sigma^2} q'(N_n)\to 0$ if and only if $N_n \gg n^{-1/(2-2\alpha)}$. 
\end{lemma} 
\noindent The proof of the lemma is sketched in Appendix~\ref{app:scales}. The critical scale $n^{-1/(2-\alpha)}$ is explained by a simple scaling relation: for $N_n = k n^{1/(2-\alpha)}$, we have 
\be  \label{eq:stretched-scaling}
	f_{n}\bigl(y n^{1/(2-\alpha)}\bigr) = n^{\alpha/(2-\alpha)} \Bigl( y^\alpha +  \frac{(k-y)^2}{2 \sigma^2} \Bigr) - \log c.
\ee
A careful examination of the expressions in Theorem~\ref{thm:critical} shows that the first summand dominates if $N_n\leq (1-\delta)N_n^{**}$ while the second dominates if $N_n\geq (1+\delta)N_n^{**}$ for some $\delta>0$. In particular, 
Theorem~\ref{thm:moderate} extends to $N_n\leq (1-\delta )N_n^{**}$, which corresponds to Theorem~1 in~\cite{nagaev}. Theorem~\ref{thm:critical} for $N_n \sim N_n^{**}$ is Theorem~4 in~\cite{nagaev}. For $N_n \geq (1+\delta) N_n^{**}$, we have 
\be 
	\P(S_n = n\mu + N_n) \sim \frac{1}{\sqrt{1- n \sigma^2 q''(x_{nr}) }} \e{- f_{nr}(x_{nr})}.
\ee 
This regime can be divided into three cases:
\begin{enumerate}[(a)] 
	\item  $N_n\ll n^{-1/(2-2\alpha)}$ corresponds to Theorem~2 in~\cite{nagaev}.
	\item When $N_n$ is of the order of $n^{-1/(2-2\alpha)}$, the corrections from the Cram{\'e}r series are irrelevant and 
	\be \label{eq:medium-large}
		\P(S_n = \mu n + N_n) \sim n \exp\bigl( - f_n(x_n)\bigr),
	\ee
	i.e., we may choose $r=0$. This corresponds to Theorem~6 from the erratum~\cite{nagaev-erratum}, replacing Theorem~5 in the original article~\cite{nagaev}. The statement actually extends to $N_n \gg n^{-1/(3-3\alpha)}$.  Indeed $ q'(x_{nr}) \sim (N_n - x_{nr})/(n\sigma^2) = o(N_n/n)$ yields $N_n - x_{nr} = (1+o(1))n \sigma^2 \alpha N_n^{\alpha -1} $ and 
	\be
		f_{nr}(x_{nr}) = f_n(x_{nr}) + O\bigl( n N_n^{- (3 - 3\alpha)}\bigr), 
	\ee
	and one can check that $f_n(x_{nr}) = f_n(x_n) +o(1)$ if $N_n\gg n^{-1/(3-3\alpha)}$. 
	\item $N_n\gg n^{-1/(2-2\alpha)}$ corresponds to Theorem~3 in~\cite{nagaev} and our Corollary~\ref{cor:big-jump}. 
\end{enumerate} 

\subsection{Application to logarithmic hazard functions} \label{sec:logarithmic}

Here we specialize to
\be \label{eq:logarithmic}
	p(k) = c\exp\bigl( - (\log k)^\beta\bigr),\quad q(k) = - \log c+ (\log k)^\beta 
\ee
with $\beta>2$.\footnote{For $\beta \in (1,2]$,  $x q'(x)= \beta (\log x)^{\beta-1} \to \infty$ but the stronger condition $x q'(x)/\log x\to \infty$ from Assumption~\ref{ass:convexity}(ii) fails. 
We suspect that this restriction is technical and could be lifted with more detailed estimates, but a proof or disproof is beyond this article's scope.
} 

\begin{lemma} \label{lem:oklog}
	The weights~\eqref{eq:logarithmic} satisfy Assumptions~\ref{ass:convexity}--~\ref{ass:analyticity2}. Moreover
	$$	N_n^*  \sim 2 x_n^*\sim  2 \sqrt{2^{1-\beta} \beta n \sigma^2 (\log n)^{\beta-1}}, 
			\quad  N_n^{**}  \sim \sqrt{2 n  \sigma^2 (\log n)^\beta}, 
		$$
		and $\sqrt{n\sigma^2} q'(N_n)\to 0$ if and only if $N_n\gg \sqrt{n}  (\log n)^{\beta-1}$.  
\end{lemma} 

\noindent The lemma is proven Appendix~\ref{app:scales}. Notice that unlike the stretched exponential case (Lemma~\ref{lem:stretched}), $N_n^{**}$ is much larger than $N_n^*$. The scaling relation~\eqref{eq:stretched-scaling} is modified as follows: for $N_n = k \sqrt{n\sigma^2 (\log n)^{\beta-1}} \sim k x_n^*$, we have 
\be
	f_n(y x_n^*) = (\log x_n^*)^\beta + \beta (\log n)^{\beta-1} \Bigl(\log y+ \frac{(k-y)^2}{2}\Bigr) + o\Bigl( (\log n)^{\beta-1}\Bigr)
\ee 
and $(\log x_n^*)^\beta \sim (\log n)^\beta \sim (\log N_n)^\beta$. 

Our results may now be applied to obtain a sharp boundary for the big-jump domain. 
\begin{theorem}\label{thm:loghazard}
	Let $p(k)$ be as in Eq.~\eqref{eq:logarithmic} with $\beta >2$ and $N_n \gg \sqrt{n}$.  Then $\P(S_n = \mu n + N_n) \sim n \P(X=N_n)$ if and only if $N_n \gg \sqrt{n}(\log n)^{\beta-1}$. 
\end{theorem}

\noindent The ``if'' part of the theorem is actually a special case of \cite[Theorem~8.2]{denisov-dieker-shneer} and as such not new. The ``only if'' part shows that the boundary derived in~\cite{denisov-dieker-shneer} is in fact sharp. 

\begin{proof} [Proof Theorem \ref{thm:loghazard}]
	Let $I_n:= \sqrt{n\sigma^2}(\log n)^{\beta-1}$ and notice $I_n \gg N_n^{**}$ for $\beta >2$. Suppose that $N_n \gg I_n$. Then we have, in particular, $N_n \gg N_n^{**}$. Write $N_n = \alpha_n N_n^{**}$ with $\alpha_n \to \infty$. Then 
	\be \label{eq:logestimate}
		\frac{N_n^2/(2 n \sigma^2)}{(\log N_n)^\beta}= \alpha_n^2 \Bigl( 1 + o(1) \frac{\log \alpha_n}{\log n}\Bigr)^{-\beta} \geq \frac{\alpha_n^2}{(\log \alpha_n)^\beta} \to \infty
	\ee
	and it follows from Corollary~\ref{cor:big-jump} that $\P(S_n = n\mu + N_n )\sim n \P(X=n)$, hence $N_n \gg N_n^{**}$ is indeed a sufficient condition. In order to check that it is necessary, we treat the case $N_n = O(I_n)$ with Theorems~\ref{thm:moderate} and~\ref{thm:critical}. 
	
	\emph{Case 1}: $N_n^* \ll N_n = O(I_n)$. In Theorem~\ref{thm:critical} we obtain a lower bound by neglecting  the first contribution and estimating $1- n \sigma^2 |q''(x_{nr})| \leq 1$. Combining with Proposition~\ref{prop:truncated}, we find 
	\begin{multline}
		\P(S_n = n \mu + N_n) \geq  n \exp( - f_{nr}(x_{nr}) + o(1)) \\
			\geq n \exp\Bigl( - q(N_n) + (1+o(1)) n \sigma^2 q'(N_n)^2 + o(1)\Bigr)
	\end{multline}
	hence in view of $N_n = O(I_n)$ and Lemma~\ref{lem:oklog}
	\be
		\liminf_{n\to \infty } \frac{\P(S_n = n \mu + N_n)}{n \P(X=N_n)} \geq \liminf_{n\to \infty }\exp\bigl( (1+o(1) n \sigma^2 q'(N_n)^2 )\bigr) >1. 
	\ee
	
	\emph{Case 2}: $N_n = O(N_n^*)$. Then we have in particular $N_n= o(N_n^{**})$. Write 
	 $N_n = \alpha_n N_n^{**}$ with $\alpha_n \to 0$, then 
	\be 
		\frac{N_n^2}{2 n \sigma^2 (\log N_n)^\beta} \leq \frac{\alpha_n^2 (\log n)^\beta}{(\log \sqrt{n})^\beta}\to 0 
	\ee
	and Theorems~\ref{thm:moderate} and~\ref{thm:critical} show
	\be
		\log \frac{\P(S_n = n \mu + N_n)}{n \P(X = N_n)} \geq - \frac{N_n^2}{2 n \sigma^2 }
		+ (\log N_n)^\beta + O(\log n)\to \infty.
	\ee
\end{proof}


\section{Proof strategy} \label{sec:strategy}

Here we explain the strategy for the proof of Theorems~\ref{thm:moderate}--\ref{thm:large}. We focus on the case $N_n = o(n)$ and Theorem~\ref{thm:critical}.  Set $m = \mu n + N_n$. 
We start from the observation that $\P(S_n = m)$ is equal to $[z^m]\, G(z)^m$, the coefficient of $z^m$ in the expansion of $G(z)^m$, which in turn is given by contour integrals 
\be
\label{Voila Cauchy}
   [z^m] G(z)^n = \frac{1}{2\pi \ii} \oint \frac{G(z)^n}{z^m} \frac{\dd z}{z} 
    = \frac{1}{2\pi \ii} \int_{0}^{2 \pi \ii} \e{n \phi(t) - m t} \dd t. 
\ee
The contour integral can be taken over any circle centered at the origin with radius $r\leq 1$.
A steepest descent ansatz would look for a point $z_n$ such that $z_n G'(z_n) = m/n = \mu+ N_n/n$, or $\eta_n$ with $\phi'(\eta_n)  =\mu + N_n/n$, and then integrate over $|z| = |z_n|$ (or $\Re t = \Re \eta_n$). However in the regime $m/n > G'(1)= \mu$ that we investigate there is no such point, and instead we follow an approach that is  in the spirit of singularity analysis~\cite{FS} but with several novel ingredients. 
Crucially, the generating function $G(z)$ \emph{does not} fall into the class of functions which Flajolet and Sedgwick call ``amenable to singularity analysis'' \cite{FS}[Chapter VI]. \\

\textbf{Step 1: Analytic extension to slit plane.} Observe that $G(z)$ has an analytic extension to the slit plane $\bbC\backslash [1,\infty)$. This is proven with the help of the Lindel{\"o}f integral \cite{Lind,FGS}, see Proposition~\ref{prop:lindeloef}. The key ingredient here is that $p(\zeta)$ 
is analytic in a complex half-plane containing the integers $k\in \bbN$ and growth slower than $\exp( (\pi - \eps)|\zeta|)$ as $\zeta \to \infty$.   \\

\textbf{Step 2: Behavior near the dominant singularity and along the slit.} 
The Lindel{\"o}f integral actually shows that the analytic extension $G(z)$ has well-defined limits as $z$ approaches the slit $[1,\infty)$  from above or below, i.e., the limits $\lim_{\eps\searrow 0} G(\e{t} + \ii \eps)$ and $\lim_{\eps\searrow 0} G(\e{t}- \ii \eps)$ exist for all $t\in \bbR$.   Moreover the imaginary part along the slit is given by a Bromwich integral, 
\be \label{eq:bromwich}
  \lim_{\eps \searrow 0} \Im  G(\e{t}+ \ii \eps) = \frac{1}{2\ii} \int_{1/2 - \ii \infty}^{1/2+ \ii\infty} \e{t \zeta} p(\zeta) \dd \zeta \quad (t\geq 0).
\ee 
 The line of integration $\Re \zeta = 1/2$ can be replaced by any other line $\Re \zeta = x>0$. 
As $t \searrow 0$, the imaginary part vanishes faster than any power of $t$, whereas the real part can be approximated to arbitrarily high order by a Taylor polynomial. \\

\begin{centering}
\begin{figure}[htb]
\begin{picture}(0,0)%
\includegraphics{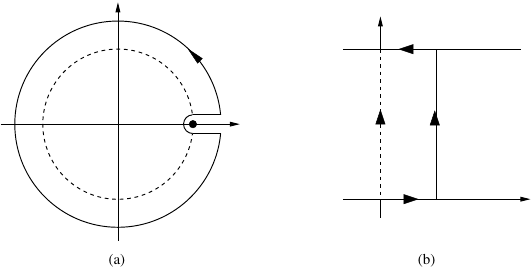}%
\end{picture}%
\setlength{\unitlength}{1973sp}%
\begingroup\makeatletter\ifx\SetFigFont\undefined%
\gdef\SetFigFont#1#2#3#4#5{%
  \reset@font\fontsize{#1}{#2pt}%
  \fontfamily{#3}\fontseries{#4}\fontshape{#5}%
  \selectfont}%
\fi\endgroup%
\begin{picture}(8499,4275)(1114,-5074)
\put(3976,-1336){\makebox(0,0)[lb]{\smash{{\SetFigFont{7}{8.4}{\rmdefault}{\mddefault}{\updefault}{\color[rgb]{0,0,0}$|z| = \e{\varepsilon}$}%
}}}}
\put(3226,-2011){\makebox(0,0)[lb]{\smash{{\SetFigFont{7}{8.4}{\rmdefault}{\mddefault}{\updefault}{\color[rgb]{0,0,0}$|z|=1$}%
}}}}
\put(8026,-4261){\makebox(0,0)[lb]{\smash{{\SetFigFont{8}{9.6}{\rmdefault}{\mddefault}{\updefault}{\color[rgb]{0,0,0}$\varepsilon$}%
}}}}
\put(6976,-4186){\makebox(0,0)[lb]{\smash{{\SetFigFont{8}{9.6}{\rmdefault}{\mddefault}{\updefault}{\color[rgb]{0,0,0}$0$}%
}}}}
\put(6751,-1486){\makebox(0,0)[lb]{\smash{{\SetFigFont{8}{9.6}{\rmdefault}{\mddefault}{\updefault}{\color[rgb]{0,0,0}$2\pi\ii$}%
}}}}
\end{picture}%
\caption{Contour integrals in the $z$-plane (a) and in the $t$-plane (b).  Dotted lines: Eq. \eqref{Voila Cauchy}. Solid lines: deformed contours in Eq.\ \eqref{eq:rectangle}. Recall that $z = \e{t}$. Later $\eps= \eta_n$ will be chosen in a judicious way. For $N_n\ll N_n^{**}$, the dominant contribution should come from the horizontal pieces of the deformed contour in the $t$-plane. For $N_n\gg N_n^{**}$, the dominant contribution should instead be from the vertical line. 
}
\label{fig contours}
\end{figure}
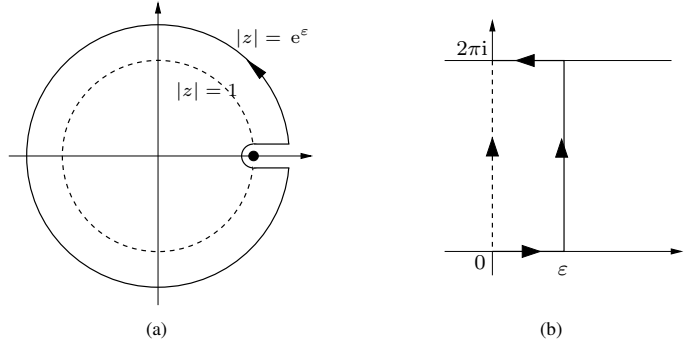
\end{centering}

\textbf{Step 3: Contour integrals.} We may now deform the contour of integration: in the $z$-plane, we replace the circle of radius $1$ by a Hankel-type contour consisting of a circle of radius $\e{\eps}$ and a piece hugging the segment $[1,\eps)$, see Figure~\ref{fig contours}. In the $t$-plane, we replace the vertical segment joining $0$ and $2\pi \ii$ by the three other sides of the rectangle with corners $0$, $\eps$, $\eps + 2 \pi \ii$, $2 \pi \ii$. This yields 
\begin{multline} \label{eq:rectangle}
  [z^m] G(z)^n   = \frac{1}{2\pi \ii}\Bigl( \int_0^\eps \e{n \phi(t) - mt } \dd t + \int_0^{2\pi} \e{n \phi(\eps+\ii \theta) - m(\eps + \ii \theta) } \ii \dd \theta \\
   -  \int_0^\eps \e{n \phi(t+2 \pi \ii) - m(t+ 2\pi \ii) }\dd t \Bigr).
\end{multline}
We focus on $N_n = o(n)$ and choose $\eps = \eta_n$ as the solution of 
\be \label{eq:etan}
	\Re\varphi'(\eta_n) = \frac{m}{n} =  \mu + \frac{N_n}{n}.
\ee
Notice $\eta_n\sim N_n/(n\sigma^2)$. 
Using the identities 
\be
   G(\overline{z})= \overline{G(z)}, \quad \forall t\geq 0:\ \phi(t+2\pi \ii) = \overline{\phi(t)}, 
\ee
Eq.~\eqref{eq:rectangle} becomes 
\be \label{eq:hv}
	[z^m] G(z)^n = H_n + V_n
\ee
with 
\be \label{eq:hvdef}
\begin{aligned}
   H_n & = \frac{1}{\pi} \int_0^{\eta_n} \e{n \Re \phi(t) - m t} \sin \bigl(n \Im \phi(t)\bigr) \dd t \\
   V_n & =  \frac{1}{\pi} \int_0^{\pi}\e{n \Re \phi(\eta_n + \ii \theta) - m \eta_n}\cos \bigl(n \Im \phi(\eta_n+\ii \theta) \bigr) \dd \theta.
  \end{aligned}
\ee
Standard arguments show that the dominant contribution to $V_n$ come from small $\theta$. Since $\Im G(\e{t}) \to 0$ faster than any power of $t$ as $t\searrow 0$ and 
\be
   \Im G(\e{t}) = \Im \e{\phi(t)} = \e{\Re \phi(t)} \Im \phi(t) \sim \Im \phi(t), 
\ee
we may drop the trigonometric functions from Eq.~\eqref{eq:hvdef} and find 
\be
	H_n  \sim \frac{n}{\pi} \int_0^{\eta_n} \e{n \Re \phi(t) - m t}  \Im G(\e{t})\dd t,\qquad
	V_n  \sim \frac{1}{\pi} \int_0^{\pi}\e{n \Re \phi(\eta_n + \ii \theta) - m \eta_n} \dd \theta.
\ee
The vertical contribution is evaluated with the help of a Gaussian approximation around $\theta =0$, which yields 
\be \label{eq:veval}
	V_n \sim \frac{1}{\sqrt{2\pi n \Re \varphi''(\eta_n)}} \e{n \Re \varphi(\eta_n) - m \eta_n}. 
\ee
With Definition~\ref{def:cramer}, we recognize in Eq.~\eqref{eq:veval} the asymptotic expression from Theorem~\ref{thm:moderate} and obtain
\be \label{eq:vnfinal}
	V_n \sim \frac{1}{\sqrt{2\pi n \sigma^2}}  \exp\Bigl( - \frac{N_n^2}{2n\sigma^2}+ \frac{N_n^3}{n^2} \sum_{j=0}^{r-1} \lambda_j \Bigl( \frac{N_n}{n}\Bigr)^j\Bigr).  
\ee
The evaluation of $H_n$ is more involved. As a preliminary step,  we express $\Im G(\e{t})$ through the Bromwich integral and find 
\be \label{eq:dxidt}
	H_n \sim \frac{n}{2 \pi \ii} \int_0^{\eta_n} \Bigl(\int_{1/2- \ii \infty}^{1/2+\ii \infty}  \e{\Phi_n(t,\zeta)} \dd \zeta \Bigr) \dd t
\ee
with $\Phi_n$ as in Eq.~\eqref{eq:sndef}. 
 \\

\textbf{Step 4: Critical points of $\Phi_n(t,\zeta)$.} 
In order to apply a Gaussian approximation to the bivariate integral~\eqref{eq:dxidt}, we look for a critical points $(t_n,\zeta_n)$ of $\Phi_n$ with $t_n \in (0,\eta_n)$ and $\zeta_n \in (0,\infty)$. 
The gradient $\nabla \Phi_n(t_n,\zeta_n)$ vanishes if and only if 
\be \label{eq:critsysta}
\left\{ 
\begin{aligned} 
  t_n & = q'(\zeta_n),\\
  \Re \bigl( \phi'(t_n) - \phi'(0)\bigr) & = \frac{N_n - \zeta_n}{n}.
\end{aligned}
\right.
\ee
Since $\Re\varphi'(t) = \mu + \sigma^2 t+ O(t^2)$ as $t\searrow 0$, Eq.~\eqref{eq:critsysta} implies 
\be
	q'(\zeta_n) \sim \frac{N_n- \zeta_n}{n\sigma^2}. 
\ee
We recognize the equation for the critical points of $f_n$. Lemma~\ref{lem:avar} suggests the following: for $N_n\ll N_n^*$, there should be no critical point, for $N_n^* \ll N_n\ll n$, there should be two. Let us focus on the latter case and label the critical points as $(t_n,\zeta_n)$ and $(t'_n,\zeta'_n)$ with $t_n< t'_n$. 
In view of Lemmas~\ref{lem:avar} and~\ref{lem:crit-magnitudes}, we expect $\zeta_n \approx x_n$ and $\zeta'_n\approx x'_n$, hence 
\be \label{eq:ea}
      t_n \sim q'(N_n),\quad \zeta_n = N_n + O(N_n^*)
\ee
and $\zeta'_n < x_n^*< \zeta_n$. The Hessian of $\Phi_n$ is 
\be
   {\rm Hess}\, \Phi_n(t,\zeta) = \begin{pmatrix} 
    n  \Re \phi''(t) & 1 \\
     1 & -q''(\zeta)
     \end{pmatrix}.
\ee
Using again Lemma~\ref{lem:crit-magnitudes} and $\zeta_n \approx x_n$, we expect 
\be \label{eq:phinhe}
  \det {\rm Hess}\, \Phi_n(t_n,\zeta_n) = -1- \bigl(1+o(1)\bigr) n \sigma^2 q''(\zeta_n) = -1 +O\Bigl(\frac{N_n^*}{N_n}\Bigr)<0 
\ee
thus $(t_n,\zeta_n)$ is a saddle point. (More precisely, it is a saddle point of $\Re \Phi_n$, but the abuse of terminology is natural and not problematic in our context.) On the other hand $\zeta'_n\approx x'_n< x_n^*$ with $1 + n \sigma^2 q'(x_n^*) =0$ by definition of $x_n^*$, so  we expect 
\be
	\det {\rm Hess}\, \Phi_n(t'_n,\zeta'_n) = -1- \bigl(1+o(1)\bigr) n \sigma^2 q''(\zeta'_n) >0. \\
\ee

\textbf{Step 5: Gaussian approximation for $H_n$.} 
In order to evaluate the double integral in Eq.~\eqref{eq:dxidt}, we use a good change of variables and a Gaussian approximation. Let $\zeta(t)$ be the solution of $q'(\zeta)= t$, so that $\partial_\zeta \Phi_n(t,\zeta) = 0$ if and only if $\zeta = \zeta(t)$. It is convenient to deform the contour and integrate along $\Re \zeta = \zeta(t)$ instead of $\Re \zeta = 1/2$. The integral becomes 
\be \label{eq:dsdt}
H_n \sim \frac{n}{2 \pi} \int_0^{\eta_n} \Bigl( \int_{-\infty}^\infty\e{\Phi_n(t,\zeta(t) + \ii s)} \dd s \Bigr) \dd t.
\ee
A straightforward computation shows that 
$F_n(t,s) = S_n(t,\zeta(t)+\ii s)$, a function of two real variables $t,s$, has a critical point at $(t_n,0)$ with positive definite Hessian 
\be
   {\rm Hess}\, F_n(t_n,0) = \begin{pmatrix} 
         \beta_n & 0 \\
         0 & - \partial_\zeta^2 \Phi_n(t_n,\zeta_n)
     \end{pmatrix}, \quad \beta_n = \frac{\det {\rm Hess}\, \Phi_n(t_n,\zeta_n)}{\partial_\zeta^2 \Phi_n(t_n,\zeta_n)}, 
\ee
see Lemma~\ref{lem:Fn}. $F_n:(0,\eta_n)\times \R\to \C$ has another critical point at $(t'_n,0)$, with negative determinant of the Hessian; later we show that it does not contribute to the integral. The evaluation of $H_n$ is concluded by replacing the double integral~\eqref{eq:dsdt} by the integral of the Gaussian approximation around $(t_n,0)$ which yields 
\be 
  H_n \sim \frac{n}{2\pi} \sqrt{\frac{(2\pi)^2}{\det {\rm Hess}\, F_n(t_n,0)}} \e{F_n(t_n,0)}
= \frac{n}{\sqrt{|\det {\rm Hess}\, \Phi_n(t_n,\zeta_n)|}} \e{\Phi_n(t_n,\zeta_n)}. \\
\ee
The argument leading leading to Eq.~\eqref{eq:phitrunc} and Eq.~\eqref{eq:phinhe} show 
\be \label{eq:hnfinal}
	H_n \sim \frac{n}{\sqrt{1 - n \sigma^2 q''(x_{nr})}} \e{- f_{nr}(x_{nr})}.
\ee

\section{Analytic continuation. Lindel{\"o}f and Bromwich integrals} \label{sec:lindeloef}

Here we take care of steps 1 and 2, starting from Assumption~\ref{ass:analyticity}.
For concreteness' sake we write down the results for $b=1/2$; they apply for general $b$ with straight-forward modifications. 
Define
\be \label{eq:lindeloef}
   \Lambda(w) = - \frac{1}{2\pi \ii} \int_{1/2 - \ii \infty}^{1/2+ \ii \infty}  p(\zeta) w^\zeta \frac{ \pi}{\sin \pi \zeta} \dd \zeta \qquad (w \in \bbC \backslash (-\infty,0]),
\ee
the Lindel{\"o}f integral with symbol $p(\zeta)$. 

\begin{prop}[\cite{Lind}]\hfill \label{prop:lindeloef}
   \begin{enumerate}                   
   \item [(a)]  $\Lambda(w)$ is analytic in the slit plane $\bbC \backslash (-\infty,-1]$. 
   \item [(b)] In the unit disk $|w| \leq 1$, 
         \begin{equation*} 
            \Lambda(w) = \sum_{k=1}^\infty p(k) (-w)^k = G(-w).
         \end{equation*} 
   \end{enumerate}
\end{prop}
A detailed proof and many additional properties of $\Lambda(w)$ can be found in~\cite{FGS}. Proposition~\ref{prop:lindeloef} shows right away that $G(z)= \Lambda(-z)$ has an analytic continuation from the unit disk to the open slit plane $\bbC\backslash [1,\infty)$. We use the same letter $G(z)$ for the analytic continuation, and set $\varphi(t) = {\rm Log}\, G(\e{t})$ with ${\rm Log}$ the principal branch of the logarithm, and $\Im t \in [0,2\pi)$. We prove the following additional properties of $G(z)$ and $\varphi(t)$.

\begin{theorem}  \label{thm:bval}\hfill 
  \begin{enumerate} 
    \item [(a)]   The boundary value  $G(\e{t}) = \lim_{\eps \searrow 0} G(\e{t}+ \ii \eps)$ exists for all $t\in \bbR$ and is a smooth function of $t \in \bbR$. 
    \item [(b)] The imaginary part $\Im G(\e{t})$, $t\geq 0$ is given by the Bromwich integral~\eqref{eq:bromwich}. 
      \item [(c)]  $\varphi(t)$ is well-defined and smooth in a neighbourhood of the origin; the derivatives $\kappa_j = \varphi^{(j)}(0)$ are real.  As 
 $t\to 0$ in the  strip $\Im t \in [0,2\pi)$, we have 
      \begin{equation*}
          \varphi(t) = {\rm Log}\,  G(\e{t}) =  \sum_{j=1}^r \kappa_j \frac{t^j}{j!} + O(t^{r+1}). 
       \end{equation*} 
      to arbitrarily high order $r$.
\end{enumerate}
\end{theorem}
In (c) $z = \e{t}$ is allowed to approach the slit $[1,\infty)$ as fast as we like; we may even take $t$ real. Because the coefficients $\kappa_j$ are real, we find in particular that $\Im G(\e{t})$ vanishes faster than any power of $t$ as $t\to 0$, $t\in \bbR$. 

\begin{proof}[Proof of Theorem~\ref{thm:bval}]
  For $u\in \bbC$ in the closed strip $\Im u \in [-\pi,\pi]$, define 
  \be 
    L(u) = - \frac{1}{2\pi \ii} \int_{1/2 - \ii \infty}^{1/2+ \ii \infty}  p(\zeta) \e{\zeta u} \frac{ \pi}{\sin \pi \zeta} \dd \zeta.
  \ee
  When $\Im u$ is in the open strip $\Im u \in (-\pi,\pi)$, we have $w= \e{u} \in \bbC \backslash (-\infty,0]$ and $L(u) = \Lambda(\e{u})$. Along the vertical line $\Re \zeta = 1/2$, we have 
  \be \label{eq:uipb}
      \Bigl| \frac{\exp(\zeta u)}{\sin (\pi \zeta)}\Bigr| = 2 \e{\Re u/2} \frac{\exp( -s \Im u)}{\exp(\pi s)+ \exp(- \pi s)} \leq 2 \e{\Re u/2} \quad (\zeta = \frac{1}{2} + \ii s)
  \ee
  By Assumption~\ref{ass:analyticity}(ii), 
  since $\zeta^k p(\zeta)$ is integrable along $\Re\zeta = 1/2$. 
 Eq.~\eqref{eq:uipb} then shows that the integral defining $L(u)$ is absolutely convergent, and it stays absolutely convergent if we replace the symbol $p(\zeta)$ by $\zeta^k p(\zeta)$. Standard arguments for parameter-dependent integrals then show that $L(u)$ is continuous on the closed strip, differentiable in the open strip, and we may exchange differentiation, limits, and integration, which shows that the restriction of $L$ to the boundaries $\Im u = \pm \pi$ yield smooth functions. 

When $z\to \e{t} \in [1,\infty)$ along $\Im z >0$, we have $w = - z \to -\e{t}$ along $\Im w <0$. Thus we may write $w= \e{u}$ with $\Re u \to t$ and $\Im u ={\rm arg}\, w \searrow - \pi$. Therefore 
\be 
  \lim_{\eps \searrow 0} G(\e{t} + \ii \eps) = L(t - \ii \pi) 
     =  \ii \int_{-\infty}^\infty p\bigl(\tfrac{1}{2} + \ii s\bigr)\e{(1/2 + \ii s) t} \frac{\exp(  \pi s)}{\exp(\pi s) + \exp( - \pi s)} \dd s. 
\ee
This proves the existence of the limit and, in view of the above mentioned properties of $L(u)$, the smoothness as a function of $t$. 
The complex conjugate is 
\begin{multline}
    - \ii \int_{-\infty}^\infty p\bigl(\tfrac{1}{2} - \ii s\bigr)\e{(1/2 - \ii s) t} \frac{\exp(  \pi s)}{\exp(\pi s) + \exp( - \pi s)} \dd s \\
      =  \ii \int_{-\infty}^\infty p\bigl(\tfrac{1}{2} + \ii s\bigr)\e{(1/2 + \ii s) t} \frac{\exp( - \pi s)}{\exp(\pi s) + \exp( - \pi s)} \dd s.
\end{multline}
Therefore 
\begin{align*}
  \Im G(\e{t}) =  \frac{1}{2} \int_{-\infty}^\infty p\bigl(\tfrac{1}{2} + \ii s\bigr)\e{(1/2 + \ii s) t} \dd s =  \frac{1}{2\ii}\int_{1/2- \ii \infty}^{1/2+\ii \infty} p(\zeta) \e{t \zeta} \dd \zeta. 
\end{align*} 
This proves (b). For (c), consider first real $t\in \bbR$. We have already checked (a) hence $G(\e{t})$ is in $C^\infty(\bbR)$. It is real and strictly positive for $t \leq 0$ (this follows from the series representation and $p(k)>0$), and non-zero though possibly complex-valued for sufficiently small $t>0$. Therefore $\phi(t)= \log G(\e{t})$ is well-defined and smooth in some interval $(-\infty,\delta)$, $\delta>0$, and real-valued for $t\leq 0$. 
 In particular, the derivatives $\kappa_j= \phi^{(j)}(0)$ exist and are real, and $\phi(t)$ can be approximated to arbitrarily high order by Taylor polynomials. The extension to complex $t$, $\Im t\in [0,2\pi)$, follows again from the smoothness of $L(u)$ in the closed strip $\Im u \in [-\pi,\pi]$. 
\end{proof}

Theorem~\ref{thm:bval}(b) has an interesting consequence. Eq.~\eqref{eq:bromwich} is, up to a factor $\pi$, the formula for the inverse Laplace transform, therefore 
\be \label{eq:laplace}
 p(\lambda)  = \pi \int_0^\infty \e{- t \lambda} \Im G(\e{t}) \dd t \quad (\Re \lambda >0). 
\ee
In the special case of stretched exponential weights, we can draw on an extensive literature as
$\exp( - \lambda^\alpha)$ is known to be the Laplace transform of a probability density, an \emph{$\alpha$-stable law}~\cite{Pol}. For $\alpha =1/2$ \cite{Doetsch}
\be \label{eq:half}
  \Im G(\e{t}) = \frac{c\sqrt{\pi}}{2 t^{3/2}}\e{-1/(4t)} \quad (t \geq 0).
\ee
For general $\alpha \in (0,1)$, we have instead~\cite[Theorem 2.4.6]{ibragimov-linnik}
\be \label{eq:imsmall-stretched}
	\Im G(\e{t})\sim \frac{c}{2}\sqrt{\frac{2\pi}{(1-\alpha) \alpha ^{-1/(1-\alpha)}}} 
			\frac{\exp\bigl(- (1-\alpha) \bigl(\frac{\alpha}{t}\bigr)^{ \alpha /(1-\alpha)}
			 \bigr) } {t ^{(2-\alpha)/(2-2\alpha) }}
\ee
as $t\searrow 0$. This is proven in~\cite{ibragimov-linnik} by applying a steepest descent approach to the Bromwich integral. For general weights, Eq.~\eqref{eq:imsmall-stretched} is generalized as follows. 

Assume that $t<\lim_{x\searrow a} |q''(x)|$. By Assumption~\ref{ass:convexity}, $q''$ is strictly increasing and negative on $(a,\infty)$. By Eq.~\eqref{eq:qprimetozero}, we have $q''(x)\to 0$ as $x\to \infty$. Consequently there exists a uniquely defined $\zeta(t)$ that solves $q'(\zeta(t)) = t$. We define
\be \label{eq:psit}
	\psi(t) = t \zeta(t) - q(\zeta(t))
\ee
and note the relations
\be \label{eq:psidual}
	\psi'(t) = \zeta(t),\quad \psi''(t) = \frac{1}{q''(\zeta(t))},
\ee
so $\psi(t)$ is monotone increasing and strictly concave. Since $-\psi(-t)$ is the Legendre transform of the convex function $- q(x)$, it comes as no surprise that Assumption~\ref{ass:convexity} on large $x$ translates into information on small $t$. 

\begin{lemma} \label{lem:psit}
	The following holds:
	\begin{enumerate} 
		\item [(a)] $\lim_{t\searrow 0}t \psi'(t)/\log t = \infty$. 
		\item [(b)] $ - \psi''(t) \geq c \frac{\psi'(t)}{t}$ for some $c>0$ and all sufficiently small $t>0$. 
		\item [(c)] $0 \leq \psi'''(t) \leq C \frac{|\psi''(t)|}{t}$ for some $C>0$ and all sufficiently small $t>0$.
	\end{enumerate} 
\end{lemma} 

\noindent The lemma has been proven in~\cite[Lemma 2.2]{nagaev73}.

\begin{theorem} \label{thm:imsmall}
	As $t\searrow 0$, 
	$$
		\Im G(\e{t}) \sim \frac{1}{2} \sqrt{2 \pi |\psi''(t)|}\,\e{\psi(t)}.
	$$
\end{theorem} 

\begin{proof} 
	By Theorem~\ref{thm:bval}(b) we may start from the Bromwich representation of $\Im G(\e{t})$. The analyticity of $q(\zeta)$ allows us to replace the contour $\Re \zeta = 1/2$ by 
	$\Gamma = \Gamma_1 \cup \Gamma_2$ where 
	\be 
		\Gamma_1 = \{ \zeta \in \C \mid |\zeta|= \zeta(t),\, \Re \zeta \geq 1/2\}, \quad
		\Gamma_2 = \{ \zeta \in \C \mid |\zeta|> \zeta(t),\, \Re \zeta = 1/2\}. 
	\ee
	To lighten notation set $r= \zeta(t)$ and suppress the $t$- and $r$-dependence from the notation. Let $\theta_0 = \arcsin (1/2r)$ and notice $\theta_0\nearrow \pi/2$ as $r\to \infty$ ($t\searrow 0$). 
	For small $\theta$ we have 
		\begin{align} 
			t\, r \e{\ii \theta} - q(r \e{\ii \theta}) &= \psi(t) - \frac{1}{2} r^2 q''(r) (\e{\ii \theta} - 1)^2 + O\bigl(r^3 q'''(r)  (\e{\ii \theta} - 1)^2 \bigr)  \notag \\
				& = \psi(t) - \frac{1}{2} r^2 |q''(r)|\bigl( \theta^2 +O(\theta^3)) \label{eq:gitti}
		\end{align}
		The estimate is uniform in $r=\zeta(t)$ by Assumption~\ref{ass:analyticity2}(iii). Let $\eps(r)\searrow 0$ with $\eps(r)^2 r^2 q''(r)/\log r\to \infty$ (this is possible by Assumption~\ref{ass:convexity}), then 
		\be \label{eq:imgdo}
			\frac{1}{2 \mathrm{i}}\int_{ - \eps(r)}^{\eps(r)} \exp\bigl( t r \e{\ii \theta} - q(r\e{\ii \theta})\bigr) \ii r \e{\ii \theta} \dd \theta 
				\sim \frac{1}{2}\sqrt{2 \pi |\psi''(t)|} \e{\psi(t)}.  
		\ee
		As we veer away from $r = \zeta(t)$ along $\Gamma_1$, the real part of $\zeta t - q(\zeta)$ decreases. Indeed for $\theta\in (0,\pi/2)$
		\begin{align}
			\frac{\dd}{\dd \theta} \Re \bigl( t r\e{\mathrm{i}\theta} - q(r \e{\mathrm{i}\theta}) \bigr) & = \Re \bigl( \ii \zeta (t - q'(\zeta))\bigr) \Big|_{\zeta = r\e{\ii \theta} } \notag \\
				& = - \Im \bigl( \zeta q'(r) -\zeta q'(\zeta)\bigr)\Big|_{\zeta = r\e{\ii \theta}} \leq 0. \label{eq:imgmono}
		\end{align}
		At the very end we have used 
		Assumption~\ref{ass:analyticity2}(ii)
		It follows that 
		\begin{multline*} 
		\left|	\frac{1}{2 \mathrm{i}}\int_{ \eps(r)}^{\theta_0} \exp\bigl( t r \e{\ii \theta} - q(r\e{\ii \theta})\bigr) \ii r \e{\ii \theta} \dd \theta \right| \\ \leq \frac{r \pi}{4} \exp\Bigl( \psi(t) - \frac{1}{2}r^2 |q''(r)| \eps(r) (1+o(1))\Bigr) = o \Bigl( \exp(\psi(t)) \Bigr)
		\end{multline*}
		Taking complex conjugates, we obtain a similar estimate for the integral from $-\theta_0$ to $- \eps(r)$. Together with~\eqref{eq:imgdo} we obtain 
		\be 
			\frac{1}{2\ii} \int_{\Gamma_1} \e{t \zeta - q(\zeta)} \dd \zeta \sim \frac{1}{2}\sqrt{2 \pi |\psi''(t)|} \e{\psi(t)}. 
		\ee
		It remains to estimate the contribution from $\Gamma_2$. Because of the monotonicity~\eqref{eq:imgmono} we have 
		\be 
			\Re (t \zeta_0 - q(\zeta_0)) \leq \psi(t) - \frac{1}{2} r^2 |q''(r)| \eps(r)^2 (1 +o(1))
		\ee
		Assumption~\ref{ass:analyticity2}(ii) ensures that
	\be 
		\Bigl| \int_{\Gamma_2} \e{t \zeta - q(\zeta)} \dd \zeta \Bigr|  
			\leq  \exp\Bigl( \Re\bigl[t \zeta_0 - q(\zeta_0)\bigr] + O( \log r))\Bigr) 
				= o\Bigl( \exp( \psi(t))\Bigr). 
	\ee	
\end{proof} 

\section{Critical point and Gaussian approximation}\label{sec:saddlepoint}

Here we prove Lemmas~\ref{lem:avar}-\ref{lem:crit-magnitudes} and Proposition~\ref{prop:truncated} and we address steps 4 and~5 of the proof strategy. 

\subsection{Variational analysis of $f_n(x)$. Critical scales}\label{sec:variational}

\begin{proof} [Proof of Lemma~\ref{lem:avar}]
We treat the case $a=0$. Under Assumption~\ref{ass:convexity}(i), $q'$ is strictly convex and decreasing. Therefore 
	\be \label{eq:tangentbound}
	\begin{aligned}
		f'_n(x) & = q'(x) - \frac{N_n- x}{n \sigma^2} \\
			&  \geq q'(x_n^*)+ q''(x_n^*) (x-x_n^*) - \frac{N_n- x}{n \sigma^2} 
			 = \frac{N_n^*- N_n}{n\sigma^2}
	\end{aligned} 
	\ee
	with equality if and only if $x = x_n^*$. 
	If $N_n< N_n^*$,  we obtain $f'_n(x) >0$ on $(0,\infty)$ and parts (a) and (b) of the lemma follow right away. 
	
	 If $N_n > N_n^*$, then $f'_n(x_n^*) = (N_n^*- N_n)/(n\sigma^2) <0$ and $\lim_{x\to \infty} f'_n(x) = \infty$, so by the intermediate value theorem $f'_n$ has at least one zero in $(x_n^*,\infty)$. On the other hand
	 \be
	 	f''_n(x) = q''(x) + (n\sigma^2)^{-1}= q''(x) - q''(x_n^*) >0 \text{ on }(x_n^*,\infty)
	\ee
	 so $f'_n$ is strictly increasing and $f'_n$ has exactly one zero $x_n$ in $(x_n^*,\infty)$, moreover $f_n(x_n) = \min_{(x_n^*,\infty)} f_n$.  Since $q'(x_n)>0$ by Assumption~\ref{ass:convexity} and $q'(x_n) = (N_n-x_n)/(n\sigma^2)$, we must have $x_n<N_n$. This proves the first part of (c). 
	 
	 If in addition to $N_n>N_n^*$, we have $\limsup_{n\to \infty} N_n/(n \sigma^2) < \lim_{x\searrow a} q'(x)$, then $\lim_{x\searrow a} f'_n(x) >0$. We have already observed that $f'_n(x_n^*)<0$. By the intermediate value theorem,  $f'_n$ has at least one zero $x'_n$ in $(0,x_n^*)$. Since $f''_n(x) = q''(x) -q ''(x_n^*)<0$ on $(0,x_n^*)$, the zero is unique and corresponds to  maximizer. This completes the proof of (c).  The proof of (d) is similar to (c) and therefore omitted. 	 
\end{proof} 

\begin{proof} [Proof of Lemma~\ref{lem:breakeven}]
	We treat the case $a=0$. 
	Write $f_n(x) = I_n(x,N_n)$ with $I_n(x,y) = q(x) + [y-x]^2/[2n \sigma^2]$. For $y>N_n^*$, let $x_n(y)>x_n^*>x'_n(y)$ be the solutions of $\partial_x I_n(x,y) =0$, with $x'_n(y)$ well-defined for $N_n \leq n \sigma^2 \sup q'$ only. Notice that $x\mapsto I_n(x,y)$ is increasing in $(0,x'_n(y))$, decreasing in $(x'_n(y),x_n(y))$, and increasing in $(x_n(y),\infty)$. We have  
	\be \label{eq:breakeven1}
		\frac{\dd}{\dd y} \bigl[I_n(x_n(y),y) - I(0,y)\bigr] = \frac{y-x_n(y)}{n\sigma^2} - \frac{y}{n\sigma^2} = - \frac{x_n(y)}{n\sigma^2} < 0.
	\ee
	As $y\searrow N_n^*$ at fixed $n$, a careful examination of the proof of Lemma~\ref{lem:avar} shows $x_n(y)\searrow x_n^*$ and $x'_n(y) \nearrow x_n^*(y)$, hence 
	$I_n(x_n(y),y) \to I_n(x_n^*, N_n^*)$. 
	But $x\mapsto I_n(x,N_n^*)$ is strictly increasing on $(0,\infty)$ because for $N_n = N_n^*$, $\partial_x I_n (\cdot,N_n) = f'_n(x)\geq 0$ by Eq.~\eqref{eq:tangentbound}, hence 
	$I_n(x_n^*,N_n^*)> I_n(0,N_n^*)$ and by continuity
	\be \label{eq:breakeven2}
		\lim_{y\searrow N_n^*} \bigl[I_n(x_n(y),y) - I(0,y)\bigr] >0.
	\ee
	Assumption~\ref{ass:convexity} implies that $q(y)= o(y)$ as $y\to \infty$. It follows that 
	\be \label{eq:breakeven3}
		\lim_{y\to \infty} \bigl[I_n(y,y) - I(0,y)\bigr] = \lim_{y\to \infty}\bigl[q(y) - \frac{y^2}{2n \sigma^2}\bigr] = -\infty. 
	\ee
	Eqs.~\eqref{eq:breakeven1}--\eqref{eq:breakeven3} guarantee the existence and uniqueness of a solution $y=N_n^{**}$ to the equation $I_n(x_n(y),y) =I(0,y)$, and  (a)--(c) follow with 
 the observation $f_n(x_n)- f_n(0) = [I_n(x_n(y),y)- I_n(0,y)]|_{y=N_n}$. 
 \end{proof} 

\begin{proof}[Proof of Lemma~\ref{lem:critsequences}]
As noted after Assumption~\ref{ass:convexity}, we have 
$\lim_{x\to \infty} x^2 q''(x) = - \infty$, moreover from the definition~\eqref{eq:xnstar} of $x_n^*$ and the observation $x_n^*\to \infty$ we get 
\be 
	1 = \lim_{n\to \infty} n \sigma^2 |q''(x_n^*)| \gg \frac{n \sigma^2}{(x_n^*)^2}
\ee
hence $x_n^*\gg \sqrt{n}$. The inequality $x_n^*< N_n^*$ follows from the definition~\ref{eq:xnstar} of $N_n^*$ and the positivity of $q'$. The inequality $N_n^*< N_n^{**}$ holds true by definition of $N_n^{**}$. 
By Assumption~\ref{ass:convexity}(v) there exists $C>0$ such that $q(x)\leq C x^\alpha$ for all sufficiently large $x$. Fix $C'>C$ and $N_n \geq ( 2 C' n \sigma^2)^{1/(2-\alpha)}$. Then for large $n$, 
\be 
	\frac{q(N_n)}{N_n/(2 n \sigma^2)} \leq 2 n \sigma^2  C N_n^{\alpha-2} \leq \frac{C}{C'} <1.
\ee	
Write $C/C' = 1-\eps$. 
It follows that $f_n(N_n) = q(N_n) < (1-\eps) \frac{N_n}{n\sigma^2} 
\leq (1-\eps) (1+o(1)) f_n(a)$, and a fortiori $\min f_n \leq f_n(N_n) < f_n(a)$, which shows $N_n \geq N_n^{**}$. This proves $N_n^{**} = O(n^{1/(2-\alpha)})$ and completes the proof of the first part of the lemma.

Next suppose by contradiction that $N_n^* /x_n^*\to 1$. Then by Eq.~\eqref{eq:xnstar}
we must have $n \sigma^2 q'(x_n^*) /x_n^* \to 0$. Since $x_n^*\to \infty$,  Assumption~\ref{ass:convexity} yields $n \sigma^2 q''(x_n^*)\to 0$, contradicting the definition of $x_n^*$. Similarly, the assumption $N_n^*/x_n^* \to \infty$ leads to $n \sigma^2 q''(x_n^*) \to \infty$, contradicting again the definition of $x_n^*$. So $N_n^*/x_n^*$ stays bounded away from $1$ and from $\infty$. 
\end{proof}

\begin{proof} [Proof of Lemma~\ref{lem:crit-magnitudes}] 
	Remember $q'(x_n^*) = [N_n^* - x_n^*]/[n\sigma^2]$ by definition of $N_n^*$, and $q'(x_n) = [N_n - x_n]/[n\sigma^2]$ by definition of $x_n$. Since $q'$ is strictly decreasing we deduce 
	\be
		\frac{N_n - x_n}{n\sigma^2} = q'(x_n)< q'(x_n^*) =  \frac{N^*_n - x_n^*}{n\sigma^2} < \frac{N_n^*}{2\sigma^2},
	\ee
	so $N_n - x_n \leq N_n^*$ (see Figure~\ref{fig critical}).
	Since $q'$ is strictly convex, we have $q'(x_n^*)>q'(x_n) + q''(x_n)(x_n^*-x_n)$ hence 
	\be
		q''(x_n) >\frac{q'(x_n) - q'(x_n^*)}{x_n - x_n^*} = \frac{(N_n - x_n) - (N^*_n- x_n^*)}{n\sigma^2 (x_n - x_n^*)} = \frac{O(N_n^*)}{n \sigma^2 (N_n+ O(N_n^*))}
	\ee
	We also know that $q''(x_n) <0$, so we obtain
	\be
		 f''_n(x_n) =  q''(x_n) + \frac{1}{n\sigma^2} = \frac{1}{n\sigma^2} \Bigl( 1+ O\Bigl(\frac{N_n^*}{N_n}\Bigr) \Bigr). 
	\ee
\end{proof} 
\begin{centering}
\begin{figure}[htb]
\begin{picture}(0,0)%
\includegraphics{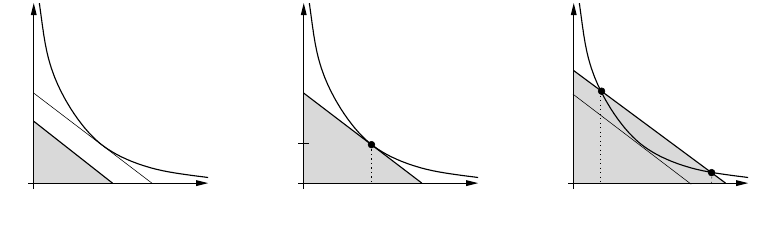}%
\end{picture}%
\setlength{\unitlength}{2368sp}%
\begingroup\makeatletter\ifx\SetFigFont\undefined%
\gdef\SetFigFont#1#2#3#4#5{%
  \reset@font\fontsize{#1}{#2pt}%
  \fontfamily{#3}\fontseries{#4}\fontshape{#5}%
  \selectfont}%
\fi\endgroup%
\begin{picture}(10221,3146)(151,-4685)
\put(6526,-4111){\makebox(0,0)[lb]{\smash{{\SetFigFont{6}{7.2}{\rmdefault}{\mddefault}{\updefault}{\color[rgb]{0,0,0}$x$}%
}}}}
\put(8101,-4186){\makebox(0,0)[lb]{\smash{{\SetFigFont{5}{6.0}{\rmdefault}{\mddefault}{\updefault}{\color[rgb]{0,0,0}$x_n'$}%
}}}}
\put(5026,-4186){\makebox(0,0)[lb]{\smash{{\SetFigFont{5}{6.0}{\rmdefault}{\mddefault}{\updefault}{\color[rgb]{0,0,0}$x_n^*$}%
}}}}
\put(5701,-4186){\makebox(0,0)[lb]{\smash{{\SetFigFont{5}{6.0}{\rmdefault}{\mddefault}{\updefault}{\color[rgb]{0,0,0}$N_n^*$}%
}}}}
\put(4351,-1636){\makebox(0,0)[lb]{\smash{{\SetFigFont{5}{6.0}{\rmdefault}{\mddefault}{\updefault}{\color[rgb]{0,0,0}$q'(x)$}%
}}}}
\put(2926,-4111){\makebox(0,0)[lb]{\smash{{\SetFigFont{6}{7.2}{\rmdefault}{\mddefault}{\updefault}{\color[rgb]{0,0,0}$x$}%
}}}}
\put(751,-1636){\makebox(0,0)[lb]{\smash{{\SetFigFont{5}{6.0}{\rmdefault}{\mddefault}{\updefault}{\color[rgb]{0,0,0}$q'(x)$}%
}}}}
\put(2101,-4186){\makebox(0,0)[lb]{\smash{{\SetFigFont{5}{6.0}{\rmdefault}{\mddefault}{\updefault}{\color[rgb]{0,0,0}$N_n^*$}%
}}}}
\put(1576,-4186){\makebox(0,0)[lb]{\smash{{\SetFigFont{5}{6.0}{\rmdefault}{\mddefault}{\updefault}{\color[rgb]{0,0,0}$N_n$}%
}}}}
\put(151,-2761){\makebox(0,0)[lb]{\smash{{\SetFigFont{5}{6.0}{\rmdefault}{\mddefault}{\updefault}{\color[rgb]{0,0,0}$\frac{N_n^*}{n\sigma^2}$}%
}}}}
\put(151,-3136){\makebox(0,0)[lb]{\smash{{\SetFigFont{5}{6.0}{\rmdefault}{\mddefault}{\updefault}{\color[rgb]{0,0,0}$\frac{N_n}{n\sigma^2}$}%
}}}}
\put(3751,-2761){\makebox(0,0)[lb]{\smash{{\SetFigFont{5}{6.0}{\rmdefault}{\mddefault}{\updefault}{\color[rgb]{0,0,0}$\frac{N_n^*}{n\sigma^2}$}%
}}}}
\put(10126,-4111){\makebox(0,0)[lb]{\smash{{\SetFigFont{6}{7.2}{\rmdefault}{\mddefault}{\updefault}{\color[rgb]{0,0,0}$x$}%
}}}}
\put(7951,-1636){\makebox(0,0)[lb]{\smash{{\SetFigFont{5}{6.0}{\rmdefault}{\mddefault}{\updefault}{\color[rgb]{0,0,0}$q'(x)$}%
}}}}
\put(9226,-4186){\makebox(0,0)[lb]{\smash{{\SetFigFont{5}{6.0}{\rmdefault}{\mddefault}{\updefault}{\color[rgb]{0,0,0}$N_n^*$}%
}}}}
\put(9751,-4186){\makebox(0,0)[lb]{\smash{{\SetFigFont{5}{6.0}{\rmdefault}{\mddefault}{\updefault}{\color[rgb]{0,0,0}$N_n$}%
}}}}
\put(9601,-3736){\makebox(0,0)[lb]{\smash{{\SetFigFont{5}{6.0}{\rmdefault}{\mddefault}{\updefault}{\color[rgb]{0,0,0}$x_n$}%
}}}}
\put(7801,-4636){\makebox(0,0)[lb]{\smash{{\SetFigFont{7}{8.4}{\rmdefault}{\mddefault}{\updefault}{\color[rgb]{0,0,0}(c) $N_n > N_n^*$ : two solutions}%
}}}}
\put(4201,-4636){\makebox(0,0)[lb]{\smash{{\SetFigFont{7}{8.4}{\rmdefault}{\mddefault}{\updefault}{\color[rgb]{0,0,0}(b) $N_n = N_n^*$ : one solution}%
}}}}
\put(601,-4636){\makebox(0,0)[lb]{\smash{{\SetFigFont{7}{8.4}{\rmdefault}{\mddefault}{\updefault}{\color[rgb]{0,0,0}(a) $N_n < N_n^*$ : no solutions}%
}}}}
\put(7351,-2836){\makebox(0,0)[lb]{\smash{{\SetFigFont{5}{6.0}{\rmdefault}{\mddefault}{\updefault}{\color[rgb]{0,0,0}$\frac{N_n^*}{n\sigma^2}$}%
}}}}
\put(7351,-2386){\makebox(0,0)[lb]{\smash{{\SetFigFont{5}{6.0}{\rmdefault}{\mddefault}{\updefault}{\color[rgb]{0,0,0}$\frac{N_n}{n\sigma^2}$}%
}}}}
\end{picture}%
\caption{Solutions to $q'(x) = (N_n - x)/(n\sigma^2)$ (=critical points of $f_n$)
  as in Lemma~\ref{lem:crit-magnitudes}.}
\label{fig critical}
\end{figure}
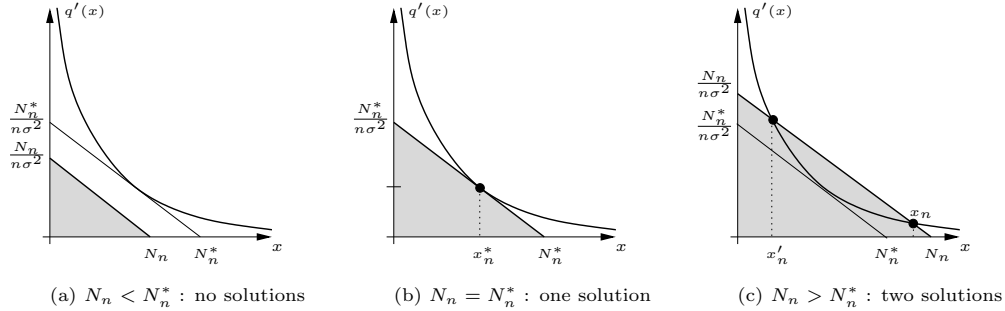
\end{centering}

For the proof of Theorem~\ref{thm:moderate} in the case $N_n \sim N_n^*$ we need the following. 

\begin{lemma} \label{lem:fncri}
	Let $N_n = N_n^*$. We have 
	\be 
		f_n(x_n^*) - \frac{(N^*_n)^2}{2 n \sigma^2 } \geq \eps \frac{(N_n^*)^2}{2 n \sigma^2 }. 
	\ee
	for some $\eps>0$ and all sufficiently large $n$, 
\end{lemma}

\begin{proof}
	Assume $a=0$ and $q(0)=0$. By Lemma~\ref{lem:breakeven} we already know $f_n(x_n^*)- f_n(0)>0$, we prove $f_n(x_n^*)-f_n(0)>\eps f_n(0)$. Since $f'_n(x_n^*)=0$ for $N_n = N_n^*$, 
	\be 
		f'_n(x) = \int_{x_n^*}^x f''_n(y) \dd y = \frac{1}{n \sigma^2} \int_{x}^{x_n^*} \bigl(n \sigma^2 |q''(y)| -1 \bigr) \dd y. 
	\ee
	For $y\geq x_n^*/2$ the integrand stays bounded away from zero, hence 
	\be 
		f'_n(x) \geq \frac{c (x_n^* - x)}{n \sigma^2}
	\ee	
	 for all $x \geq x_n^*/2$ and some $c>0$. Then 
	\be 
		f_n(x_n^*) - f_n(0) = \int_0^{x_n^*} f'_n(y) \dd y \geq \frac{c{x_n^*}^2}{8 n \sigma^2}
	\ee
	and the statement follows from $N_n^* = O(x_n^*)$ (Lemma~\ref{lem:critsequences}) 
	and $f_n(0) = \frac{(N^*_n)^2}{2 n \sigma^2 }(1+o(1))$. 
 The proof for $a>0$ is based on a similar estimate of $f_n(x_n^*)- f_n(2a)$ and therefore omitted. 
\end{proof}
\begin{proof}[Proof of Prop.~\ref{prop:truncated}]
	Fix $r\in \N_0$. We have 
	\be
		f'_{nr}(x) 
		= f'_n(x) + O\Bigl((\frac{N_n-x}{n\sigma^2})^2\Bigr). 
	\ee
	Clearly $f'_{nr}(N_n) = q'(N_n)>0$. Fix $\delta\in (0,\eps)$ and $x\in (x_n^*,(1+\delta)x_n^*)$. 
	Remembering~\eqref{eq:xnstar} and the monotonicity of $q'$, we obtain 
	\be \label{eq:trusti}
	\begin{aligned}
		f'_{nr}(x) & \leq \frac{N_n^* - x_n^*}{n \sigma^2}  - (1+o(1))\frac{N_n -(1+\delta)x_n^*}{n \sigma^2} \\
		 & = - \frac{1}{n\sigma^2}\bigl( N_n +o(N_n) - N_n^* - \delta x_n^* \bigr)
		  \leq - \frac{1}{n\sigma^2} \bigl( (1+o(1)) \eps N_n^* - \delta x_n^* \bigr)
	\end{aligned}	
	\ee
	which is eventually negative because of $x_n^*\leq N_n^*$ and $\delta <\eps$. It follows that $f'_{nr}$ does indeed have a zero $x_{nr}$ which lies between $(1+\delta)x_{n}^*$ and $N_n$. On $(x_n^*,(1+\delta)x_n^*)$ $f_{nr}$ is strictly decreasing by~\eqref{eq:trusti}, on $((1+\delta)x_n^*,N_n)$ the second derivative satisfies 
	\be
		f''_{nr}(x) = \frac{1}{n\sigma^2}\Bigl( 1- n \sigma^2 |q''(x)| + O\bigl( \frac{N_n}{n}\bigr)\Bigr)
	\ee
	which stays bounded away from $0$, hence $f_{nr}$ is strictly convex. It follows that $x_{nr}$ is the unique zero of $f'_{nr}$ and the maximizer of $f_{nr}$ in $(x_n^*,N_n)$. 
	Moreover 
	\be
		(1+o(1))\frac{N_n - x_{nr}}{n\sigma^2} = q'(x_{nr})\leq q'(x_n^*) = \frac{N_n^*- x_n^*}{n\sigma^2}
	\ee
	hence $x_{nr} = N_n - (1+o(1))n \sigma^2 q'(x_{nr}) = N_n +  O(N_n^*)$. The asymptotic expression for $f_{nr}(x_{nr})$ is easily checked. 
\end{proof}	

\subsection{Critical points}
 We look for critical points $(t_n,\zeta_n)$ with $t_n\searrow 0$ and $\zeta_n\in (a,\infty)$ of  
\be
	\Phi_n(t,\zeta) = - q(\zeta) + n \varphi(t) - mt + t \zeta
\ee
(remember $m= \mu n + N_n$).
Since our contour integrals involve integrals over $z=\e{t}$, it is 
 convenient to work with functions of a single variable $t$: for $t<\lim_{x\to a} |q''(x)|$, let 
\be
	\Psi_n(t) = \Phi_n(t,\zeta(t))  = n \varphi(t) - m t+ \psi(t).
\ee
where $\zeta(t)$ is the solution of $q'(\zeta(t))=t$ as on p.~\pageref{eq:psit}. 
Then $(t,\zeta)$ is a critical point of $\Phi_n(t,\zeta)$ if and only if $\zeta = \zeta(t)$ and $\Psi'_n(t) =0$. So instead of looking for bivariate critical points, we may look for zeros of $\Psi'_n(t)$ in $(0,\sup |q''(x)|)$. For later purpose we note the relations
\begin{align}
	\Psi'_n(t) & = n \varphi'(t) - m + \zeta(t) = - (N_n - \zeta(t)) + n \sigma^2 t(1 +O(t)) \label{eq:Psinprime}\\
	\Psi''_n(t) & = n \varphi''(t) + \psi''(t) = n \sigma^2 (1+ O(t)) + \frac{1}{q''(\zeta(t))},
\end{align} 
with $\psi(t)$ defined in~\eqref{eq:psit}. The variable $t$ and the analysis of $\Psi_n$ are in some sense dual to the variable $x$ (or $\zeta$) and the variational problem $f_n(x)=\min$ analyzed in Section~\ref{sec:variational}. The analysis becomes more involved, however, because we need to take into account correction terms from $\sum_{j\geq 3} \kappa_j t^j/j!$. 

\begin{lemma}[Inflection point of $\Psi_n$] \label{lem:inflection}
	Let $\delta>0$ such that $\inf_{(0,\delta)}\varphi''(t)>0$. Then for all sufficiently large $n$, $\Psi_n$ has an inflection point $t_n^*\in (0,\delta)$. Any inflection point satisfies $t_n^*\sim q'(x_n^*)$, and $\Psi_n$ is concave below the smallest inflection point and convex above the largest inflection point.
\end{lemma} 

\noindent If the inflection point is unique, then $\Psi_n$ is concave on $(0,t_n^*)$ and $(t_n^*,\delta)$. In general we do not know whether the inflection point is unique, however the asymptotic behavior $t_n^* \sim q'(x_n^*)$ is uniquely determined, and all statements below hold for every inflection point $t_n^*$.

\begin{remark} \label{rem:tnstar}
	It follows that $t_n^*$ is of the order of $N_n^*/n$: from Lemma~\ref{lem:inflection} and the definition of $N_n^*$, we have 
	$t_n^* \sim \frac{N_n^* - x_n^*}{n\sigma^2}$, and then Lemma~\ref{lem:critsequences} yields 
	\be 
			\delta \frac{N_n^*}{n} \leq	t_n^* \leq C \frac{N_n^*}{n}.
	\ee 
\end{remark} 

\begin{proof}[Proof of Lemma~\ref{lem:inflection}]
	As a preliminary observation, we note that any solution $t_n^*$ of $\Psi'_n(t) =0$ converges to zero: this is because $-\psi''(t^*_n) = n \varphi''(t)\geq n \inf_{(0,\delta)} \varphi'' \to \infty$. 
 
	As $t\searrow 0$ at fixed $n$, we have $\varphi''(t)\to \sigma^2$ and $\psi''(t) \to -\infty$ hence $\Psi_n''(t)\to - \infty$. On the other hand we may choose $\eps_n$ in such a way that $\eps_n\searrow 0$ and $\Psi''_n(\eps_n)\to \infty$: indeed $\psi''(q'(x_n^*)) = 1 / q''(x_n^*) = - n \sigma^2$ by definition of $x_n^*$, so choosing $\eps_n\gg q'(x_n^*)$ in such a way that $|\psi''(\eps_n)|\ll |\psi''(q'(x_n^*))|= n \sigma^2$ we find $\Psi_n''(\eps_n) = (1+o(1)) n\sigma^2 + o(n \sigma^2) \to \infty$. 
	It follows from the intermediate value theorem that $\Psi_n''(t)=0$ has a solution $t^*_n$ in $(0,\eps_n)$. It satisfies 
	\be
		\frac{q''(x_n^*)}{q''(\zeta(t_n^*))} =  - \frac{\psi''(t_n^*)}{n \sigma^2} =  \frac{\varphi''(t_n^*)}{\sigma^2} = 1 + O(t_n^*)\to 1. 
	\ee
	Assumption~\ref{ass:convexity}(iv) and its consequence~\eqref{eq:qsecvar} imply that $\zeta(t_n^*) / x_n^*\to 1$ and $t_n^* = q'(\zeta(t_n^*))\sim q'(x_n^*)$.
\end{proof} 

\noindent $\Psi'_n(t)$ is positive for small $t$ and decreasing on $(0,t_n^*)$ and increasing on $(t_n^*,\delta)$. Define $\zeta_n^* = \zeta(t_n^*)$ and notice $\zeta_n^* \sim x_n^*$ from the proof of the previous lemma. 

%

For $N_n \ll n$, let $\eta_n\sim N_n/(n\sigma^2)$ be the solution of~\eqref{eq:etan}. Notice $\Psi'_n(\eta_n) = \zeta(\eta_n)\to \infty$. 

\begin{lemma} \label{lem:crit-medium}
	Let $\delta>0$ be as in Lemma~\ref{lem:inflection}, $\eps>0$, and $t_n^*$ an inflection point of $\Psi_n$. Let $N_n\to \infty$ with $(1+\eps) N_n^*\leq N_n \ll n$. Then for sufficiently large $n$, $\Psi'_n$ has exactly two zeros in $(0,\delta)$, one zero $t_n \in (0, t_n^*)$ and another $t'_n \in (t_n^*,\eta_n)$. Set $\zeta_n := \zeta(t_n)$. We have
	$$	0 \leq N_n - \zeta_n = O(N_n^*), $$
	$ \limsup_{n\to \infty} t_n / t_n^*<1$, $\liminf_{n\to \infty} \zeta_n/x_n^*>1$, and $\liminf t'_n/\eta_n >0$.
\end{lemma} 

\begin{remark} \label{rem:tn}
	When $N_n \gg N_n^*$, we use Eq.~\eqref{eq:qprimevar}, Lemmas~\ref{lem:critsequences} and~\ref{lem:inflection} and find 
	\be 
	  	\zeta_n\sim N_n, \quad t_n = q'(\zeta_n)\sim q'(N_n) \ll q'(x_n^*)  \sim t_n^* = O\Bigl(\frac{N_n^*}{n}\Bigr)
  	\ee  	
  	hence $t_n = o(t_n^*)$. 	
	When $N_n = O(N_n^*)$, we have instead $c^{-1} t_n^* \leq t_n \leq c\, t_n^*$ for some $c>0$ and all sufficiently large $n$: The upper bound is part of Lemma~\ref{lem:crit-medium}. For the lower bound, we note that $t_n = q'(\zeta_n) \geq  q'(N_n)$ because $q'$ is decreasing and $\zeta_n \leq N_n$. Since $N_n$ is of the order of $N_n^*$, Eq.~\eqref{eq:qprimevar} shows that $q'(N_n)$ is of the order of $q'(N_n^*)$ which in turn is of the order of $q'(x_n^*)\sim t_n^*$.
\end{remark}

\begin{proof}[Proof of Lemma~\ref{lem:crit-medium}]
	We check first that $t_n^*< \eta_n$. With $C\geq 1$ as in Lemma~\ref{lem:critsequences}, we have 
	\be
		t_n^* \sim q'(x_n^*) = \frac{N_n^* - x_n^*}{n\sigma^2} \leq (1- C^{-1}) \frac{N_n^*}{n\sigma^2}
	\ee
	so 
	\be \label{eq:tstareta}
		t_n^*/\eta_n \leq (1+o(1))(1-C^{-1}) N_n^*/N_n \leq (1+o(1))\frac{1- C^{-1}}{1+\eps}
	\ee
	 is bounded away from $1$. As $t \to 0$ at fixed $n$, $\Psi'_n(t) \sim \zeta(t) \to  \infty$, and as $n\to \infty$, $\Psi'_n(\eta_n)\to \infty$.  Furthermore
	\be \label{eq:crin}
	\begin{aligned}
		\Psi'_n(t_n^*)  & = - (N_n -\zeta_n^*) + (1+o(1)) n \sigma^2 q'(\zeta_n^*)\\
			& = - \bigl(N_n - (1+o(1)) x_n^*\bigr) + (1+o(1)) n \sigma^2 q'(x_n^*) \\
			& = - \bigl(N_n - (1+o(1)) x_n^*\bigr) + (1+o(1)) \bigl(N_n^*- x_n^*\bigr) \\
			& = - (N_n - N_n^*) + o (N_n^*) \to - \infty. 
	\end{aligned} 
	\ee
	The intermediate value theorem proves the existence of a zero $t_n \in (0,t_n^*)$ and another zero $t'_n \in (t_n^*,\eta_n)$.
	 
	 Suppose by contradiction that $t_n/ t_n^*\to 1$. Then the identities $t_n = q'(\zeta_n)$, $t^*_n = q'(\zeta_n^*)$ and Eq.~\eqref{eq:qprimevar} imply $\zeta_n \sim \zeta_n^*\sim x_n^*$ and an estimate analogous to~\eqref{eq:crin} shows $\Psi'_n(t_n)\to -\infty$, in contradiction with $\Psi'_n(t_n)=0$. It follows that $t_n/t^*_n$ and $\zeta_n / x_n^*$ are  bounded away from $1$. In addition, $t_n \leq t_n^* = O(N_n^*/n)$ and
\be
	N_n - \zeta_n \sim n\sigma^2 t_n = O(N_n^*). 
\ee
For the lower bound of $t'_n$,  
	we use $t'_n \geq t_n^* \sim q'(x_n^*)$ and Assumption~\ref{ass:convexity} to get $\zeta'_n \leq (1+o(1)) x_n^*$ where $q'(\zeta'_n) = t'_n$. Since $t'_n = O(\eta_n) \to 0$, $\Psi'_n(t'_n)=0$ together with~\eqref{eq:Psinprime} yields 
	\be 
	 t'_n \sim \frac{N_n - \zeta'_n}{n \sigma^2} \sim \eta_n \Bigl( 1 - \frac{\zeta_n'}{N_n} \Bigr) \geq \eta_n\Bigl( 1-\frac{N_n^*}{N_n} \Bigr)
	\ee 
	and $N_n \geq (1+\eps) N_n^*$ implies $\liminf t'_n /\eta_n >0$. Notice that, in view of~\eqref{eq:tstareta}, we have a fortiori $\liminf t'_n /t_n^* >0$.
	
	We have actually shown that for every inflection point $t_n^*$, $t/t_n^*\leq 1$ stays bounded away from $1$. In particular, in case of non-uniqueness of $t_n^*$ we may choose $t_n^*$ as the \emph{smallest} inflection point of $\Psi_n$. Then $\Psi'_n$ is strictly increasing on $(0,t_n^*)$, consequently the zero $t_n$ is unique. A similar argument shows that $t'_n$ is unique. 
\end{proof} 

\noindent When $\liminf_{n\to \infty} N_n/n>0$, the sequence $\eta_n$ is either no longer defined or it does not converge to zero. The previous lemma is modified as follows. 

\begin{lemma} \label{lem:crit-large}
	Assume $N_n\to \infty$ with $\liminf N_n/n >0$. Then there exists $\delta_0>0$ such that for large $n$, $\Psi_n$ has exactly one critical point $t_n$ in $(0,\delta_0)$. The critical point lies in $(0,t_n^*)$ and it satisfies $0 \leq N_n -\zeta(t_n) =O(N_n^*)$.
\end{lemma}

\begin{proof} 
	The existence and uniqueness of a critical point in $(0,t_n^*)$ as well as the properties of $\zeta(t_n)$ are proven as in the previous lemma. 
	Fix $\delta_0>0$ such that $\varphi'(\delta_0) \leq \mu + \frac{1}{2} \liminf (N_n/n) =:\mu + \eps /2$. Then 
	\be
		\Psi_n'(\delta_0) = - n \Bigl(\mu+ \frac{N_n}{n} - \varphi'(\delta_0)\Bigr) + \zeta(\delta_0) \leq - n (\frac{\eps}{2}+o(1)\Bigr) + \zeta(\delta_0)\to - \infty.
	\ee
	It follows that $\Psi'_n<0$ on $(t_n^*,\delta_0)$. 
\end{proof} 

\subsection{Hessians}

Let $t_n\in (0,t_n^*)$ be the critical point of $\Psi_n(t)$, and $\zeta_n = \zeta(t_n)$. Thus $(t_n,\zeta_n)$ is a critical point of $\Phi_n(t,\zeta)$. Lemma~\ref{lem:Fn} shows 
\be \label{eq:psinsec}
	\Psi''_n(t_n) = - \frac{\det \mathrm{Hess}\, \Phi_n(t_n,\zeta_n)}{|q''(\zeta_n)|} = - \frac{1- n \Re \varphi''(t_n) |q''(\zeta_n)|}{|q''(\zeta_n)|}. 
\ee

\begin{lemma} \label{lem:hessian}
	Assume $N_n\to \infty$ with $\liminf_{n\to \infty} (N_n / N_n^*)>1$ and let $(t_n,\zeta_n)$ be the unique critical point of $\Phi_n$ in $(0,t_n^*)\times (a,\infty)$. 
	\begin{enumerate}
		\item [(a)] If $N_n\gg N_n^*$, then $\det\mathrm{Hess}\, \Phi_n(t_n,\zeta_n)\to - 1$. 
		\item [(b)] If $N_n = O(N_n^*)$, then $\det\mathrm{Hess}\, \Phi_n(t_n,\zeta_n)= - (1 - n \sigma^2 |q''(\zeta_n)|)+o(1)$ and it stays bounded away from zero. 
	\end{enumerate} 
\end{lemma} 

\begin{proof} 
	(a) If $N_n\gg N_n^*$, then by Lemmas~\ref{lem:crit-medium} and~\ref{lem:crit-large}, we have $\zeta_n = N_n + O(N_n^*) \sim N_n$, hence in particular $\zeta_n \gg N_n^*\geq x_n^*$. 
	Exploiting the monotonicity and the convexity of $q'$, we have 
	\be \label{eq:qsecco}
		0 \geq q''(\zeta_n) \geq \frac{q'(\zeta_n) - q'(x_n^*)}{\zeta_n - x_n^*} 
			 = \frac{O(N_n^*/(n\sigma^2))}{N_n (1 +o(1))} = - \frac{1}{n\sigma^2} O\Bigl( \frac{N_n^*}{N_n}\Bigr)
	\ee
	from which we get 
	\be
		\det \mathrm{Hess}\, \Phi_n(t_n,\zeta_n) = - 1 + n \sigma^2 (1+ O(t_n)) |q''(\zeta_n)| = - 1 + O\Bigl(\frac{N_n^*}{N_n}\Bigr) \to -1. 
	\ee
	(b) If $N_n =O(N_n^*)$: By Lemma~\ref{lem:crit-medium}, $\zeta_n\geq (1+\eps) x_n^*$ for some $\eps>0$. 	
	Consequently $|q''(\zeta_n)|\leq (1- \delta) |q''(x_n^*)| = (1-\delta) / (n\sigma^2)$ for some $\delta>0$ and large $n$, from which we deduce that 
	\be
		1- n \Re \varphi''(t_n) |q''(\zeta_n)|\geq 1 - (1+ O(t_n))(1- \delta) = 1- \delta + O(t_n), 
	\ee
	in particular the expression stays bounded away from zero. We also have 
	\be
		\det \mathrm{Hess}\, \Phi_n(t_n,\zeta_n) =  - (1 - n \sigma^2 |q ''(\zeta_n)|) + o (t_n) n \sigma^2 |q''(\zeta_n)|
	\ee
	Since $\zeta_n > x_n^*$, the estimate~\eqref{eq:qsecco} still holds true and
	\be
		t_n n \sigma^2 |q''(\zeta_n)| =t_n O\Bigl( \frac{N_n^*}{N_n}\Bigr) = O(t_n) \to 0.
	\ee
\end{proof} 

\subsection{Gaussian approximation} \label{sec:gaussian}

Next we address the Gaussian approximation for the evaluation of $H_n$. Because of Theorem~\ref{thm:imsmall}, we need not deal with a bivariate integral and instead may use 
\be \label{eq:hnpsi}
	\frac{1}{\pi}\int_0^{\eps_n}\e{ n\Re \varphi(t) - mt}   \Im G(\e{t}) \dd t \sim \frac{1}{\sqrt{2\pi}}  \int_0^{\eps_n} \sqrt{|\psi''(t)|} \e{\Psi_n(t)} \dd t
\ee 
as $n\to \infty$ and $\eps_n \searrow 0$. Remember that $\Psi_n(t_n) = \Phi_n(t_n,\zeta_n)$ and from Eq.~\eqref{eq:psinsec} and Lemma~\ref{lem:hessian}, 
\be \label{eq:psiratios}
	\frac{\psi_n''(t_n)}{\Psi_n''(t_n)} = \frac{1}{1 - n \sigma^2 |q''(\zeta_n)|}
\ee
with a denominator bounded away from zero. 
The following technical lemma helps estimate the prefactor $\sqrt{|\psi''(t)|}$. Set 
\be \label{eq:rndef}
	R_n(t):= \Psi_n(t) + \log \sqrt{|\psi''(t)|}. 
\ee

\begin{lemma} \label{lem:prefactor}
	Let $N_n\to \infty$ with $\liminf (N_n/N_n^*)>1$ and $t_n \in (0,t_n^*)$ the zero of $\Psi'_n(t)$ from Lemma~\ref{lem:crit-medium}. Then $R'_n$ has at least one zero $s_n \in (0,t_n)$. Moreover there exists a sequence $\delta_n\searrow 0$ such that every such zero lies in $((1-\delta_n)t_n,t_n)$. 
\end{lemma}

\begin{proof}
	As $t\searrow 0$ at fixed $n$,  using Lemma~\ref{lem:psit}, we have
	\begin{align} 
		R'_n(t) = \Psi'_n(t) + \frac{1}{2} \frac{\psi'''(t)}{\psi''(t)} & = n \bigl(\Re \varphi'(t) - \mu\bigr) - N_n  + \psi(t) + \frac{1}{2} \frac{\psi'''(t)}{\psi''(t)} \notag \\ 
			& = n (\sigma^2 t + O(t^2)) - N_n + \psi'(t) + O(1/t) \notag \\ 
			& = - N_n + o(1) + (1+o(1)) \psi'(t) \to \infty, \label{eq:rninit}
	\end{align}
	hence $t\mapsto R_n(t)$ is initially increasing. At $t=t_n$ we have 
	\be 
		R'_n(t_n) = \Psi'_n(t_n) +  \frac{1}{2} \frac{\psi'''(t_n)}{\psi''(t_n)}
			=  \frac{1}{2} \frac{\psi'''(t_n)}{\psi''(t_n)} <0. 
	\ee
	The intermediate value theorem guarantees the existence of a zero $s_n$ of $R'_n$. Set 
	\be \label{eq:yn}
		y_n:= \psi'(s_n) + \frac{1}{2} \frac{\psi'''(s_n)}{\psi''(s_n)}. 
	\ee
	In view of Lemma~\ref{lem:psit}, we have $y_n \sim \psi'(s_n)$ and by Eqs.~\eqref{eq:qprimevar} and~\eqref{eq:psidual},  $q'(y_n) \sim s_n$. 
	From $s_n \leq t_n \leq t_n^*$ and Lemma~\ref{lem:crit-magnitudes} we get that $y_n$ is larger than $x_n^*$ and actually bounded away from it. 
	By the definition of $s_n$, 
	\be 
		0 = n \bigl( \Re \varphi'(s_n) - \mu \bigr) - N_n + y_n = n \sigma^2 s_n(1 + O(s_n)) - [N_n - y_n]
	\ee
	hence 
	\be 
		\frac{N_n - y_n}{n\sigma^2} \sim s_n \sim q'(y_n). 
	\ee
	Let $x_n \in (x_n^*,N_n)$ be the solution of $q'(x_n) = (N_n- x_n)/(n \sigma^2)$. Now
	\be 
		\frac{\dd}{\dd y}\Bigl( q'(y) - \frac{N_n - y}{n\sigma^2} \Bigr)
			=q''(y) + \frac{1}{n \sigma^2} = \frac{1- n \sigma^2|q''(y)|}{n \sigma^2} 
	\ee
	and $n \sigma^2 |q''(y)|$ stays bounded away from $1$ when $y\geq (1+\delta)x_n^*$, we find that for some $c>0$, we have 
	\be
		\Bigl|q'(y_n)  - \frac{N_n - y}{n\sigma^2} \Bigr| \geq 
		\frac{ c |y_n - x_n|}{n \sigma^2}
	\ee
	hence $y_n - x_n = o(N_n - y_n) = o(N_n)$. If $N_n \gg N_n^*$, Lemma~\ref{lem:crit-magnitudes} says $x_n \sim N_n$ and we deduce $y_n - x_n =o(x_n)$. If $N_n = O(N_n^*)$, we use $x_n \geq x_n^*$ in conjunction with 
	Lemma~\ref{lem:critsequences} and find $y_n - x_n = o(N_n^*)= o(x_n^*) = o(x_n)$. 
	
	Thus we have checked that $y_n \sim x_n$, which in turn yields $s_n \sim q'(x_n)$. An entirely similar argument shows $t_n \sim q'(x_n)$, so we must have $s_n \sim t_n$. This holds for every zero in $(0,t_n)$, in particular the smallest one, and the lemma follows. 
\end{proof}

\begin{lemma} \label{lem:gaussian}  
	Let $N_n\to \infty$ with $\liminf (N_n/N_n^*)>1$ and $N_n = O(N_n^*)$. Then 
	$$\int_0^{t'_n} \sqrt{|\psi''(t)|} \e{\Psi_n(t)} \dd t 
		\sim \sqrt{\frac{2\pi}{1- n \sigma^2 q''(\zeta_n)}}\, \e{\Psi_n(t_n)}. $$
\end{lemma} 

\begin{proof} 
		By Lemma~\ref{lem:inflection} and~\ref{lem:crit-medium}, $\Psi_n$ is increasing on $(0,t_n)$ and decreasing on $(t_n,t'_n)$. We use a Gaussian approximation to $\Psi_n$ around $t_n$ and adapt~\cite[Lemma 2.1]{nagaev73}. 
		First we check that the window $(t_n -\eps_n, t_n+ \eps_n)$ contributing most to the Gaussian integral fits amply into $(0,t_n')$, i.e., 
		\be \label{eq:windowsize}
			\eps_n:= \frac{1}{\sqrt{|\Psi_n''(t_n)|}} 
			 = o \bigl( \min(t_n, t'_n - t_n)\bigr). 
		\ee			
		By Lemma~\ref{lem:crit-medium}, we have $t_n \leq (1- \delta) t_n^* \leq (1-\delta) t'_n$ for some $\delta \in (0,1)$, hence $t'_n - t_n \geq \delta t_n$ and 
		\be \label{eq:tprimet}
 			\min(t_n, t'_n - t_n) \geq \delta t_n. 
 		\ee
		Eq.~\eqref{eq:psiratios} and Lemma~\ref{lem:hessian} show that $\Psi''_n(t_n)$ is of the order of $\psi''(t_n)$, $\eps_n^2$ of the order of $1/|\psi_n''(t_n)|$. By Lemma~\ref{lem:psit}, 
		$t_n^2 |\psi_n''(t_n)| \to \infty$ hence $\eps_n^2/t_n^2 \to 0$  and~\eqref{eq:windowsize} follows. 
		The same argument shows that for 
		\be 
			c_n \to \infty \quad \text{with}\quad c_n \eps_n = o(t_n)
		\ee 
		we still get $c_n \eps_n = o \bigl( \min(t_n, t'_n - t_n)\bigr)$. 
		 Second, we observe that the prefactor $\psi''(t)$ is essentially constant on the relevant window: 
		 Because of $\psi'''(u)/ \psi''(u) = O(1/u)$ (Lemma~\ref{lem:psit}), we have
 		\be
 			\frac{\psi''(t)}{\psi''(t_n)}  = \exp\Bigl( -\int_{t_n}^t \frac{\psi'''(u)}{\psi''(u)} \dd u\Bigr)  = \exp\Bigl( O(\log \frac{t}{t_n})\Bigr) 
		\ee
		and  
		\be \label{eq:relevant-secder}
			\sup\Bigl \{ \Bigl| \frac{\psi''(t)}{\psi''(t_n)} - 1\Bigr|\, \Big|\, |t- t_n|\leq c_n \eps_n \Bigr\} = O(c_n \eps_n )\to 0.
		\ee
		Third, we note that cubic corrections can be neglected: 
		\be \label{eq:third-order}
			\Psi_n(t) = \Psi_n(t_n) + \frac{1}{2}\bigl( 1 + o(1)\bigr) \Psi_n''(t_n) (t-t_n)^2
		\ee 
		uniformly in $|t-t_n| \leq c_n \eps_n$. To this aim write, with the help of~\eqref{eq:relevant-secder}, 
		\begin{align*} 
			\Psi''_n(t) &= n \bigl( \varphi''(t_n) + O(c_n \eps_n) \bigr) + \psi''_n(t_n) (1+ O(c_n \eps_n) ) \notag \\
				&= \Psi_n''(t_n) + \psi''_n(t_n) \Bigl( O\Bigl(\frac{n c_n \eps_n}{\psi_n''(t_n)}\Bigr) + O(c_n \eps_n) \Bigr)
		\end{align*}
		Now $n / \psi''(t_n) = n q''(\zeta_n) = O(1)$ by Lemma~\ref{lem:hessian} and $\psi''(t_n) = O(\Psi_n''(t_n))$ by Eq.~\eqref{eq:psiratios} and the same lemma, hence 
		\be
			\Psi''_n(t) = (1+ O(c_n \eps_n))\Psi_n''(t_n)
		\ee
	 in $|t-t_n| \leq c_n \eps_n$ and~\eqref{eq:third-order} follows.  Eqs.~\eqref{eq:relevant-secder}, \eqref{eq:third-order} and~\eqref{eq:psiratios} yield 
		\be \label{eq:gaussian-dominant}
			\int_{t_n - c_n\eps_n}^{t_n+ c_n\eps_n} \sqrt{|\psi''(t)|} \e{\Psi_n(t)} \dd t   \sim \sqrt{\frac{ 2 \pi \psi''(t_n)}{\Psi_n''(t_n)}} 
				\sim \sqrt{\frac{2\pi}{1- n \sigma^2 q''(\zeta_n)}}\, \e{\Psi_n(t_n)}.
		\ee
		Our next task is to estimate the integral on $(0,t_n- c_n\eps_n)$ and $(t_n + c_n\eps_n,t'_n)$, taking into account that the prefactor $\sqrt{|\psi''(t)|}$ goes to infinity. On both intervals we have 
		\be \label{eq:psintails}
			\Psi_n(t) \leq \Psi_n(t_n) - \frac{1}{2}(1+O(c_n \eps_n)) c_n^2. 
		\ee
		For $t\geq t_n + c_n \eps_n$ we have $ - \psi''(t) \leq - \psi''(t_n)$ hence 
		\be \label{eq:integrand-tails}
			\sqrt{|\psi''(t)|} \exp(\Psi_n(t)) \leq \exp\Bigl( \Psi_n(t_n) - \frac{1}{2}(1+o(1)) c_n^2 + \frac{1}{2} \log |\psi''(t_n)|\Bigr).
		\ee
		By Lemma~\ref{lem:prefactor} we may choose $c_n \eps_n \searrow 0$ in such a way that $R_n$ is increasing on $(0,t_n(1- c_n \eps_n))$, which yields 
		\be 
			\Psi_n(t) + \log \sqrt{|\psi''(t)|} \leq \Psi_n(t_n- c_n \eps_n) + \log \sqrt{|\psi''(t_n - c_n \eps_n)|}. 
		\ee
		This estimate, combined with Eqs.~\eqref{eq:relevant-secder} and~\eqref{eq:psintails}, shows that~\eqref{eq:integrand-tails} holds true not only in $(t_n + c_n \eps_n, t'_n)$ but also in $(0, t_n - c_n \eps_n)$.
		Next we check that we can choose $c_n\to \infty$ so that not only $c_n \eps_n = o(t_n)$ but in fact 
		\be \label{eq:gtails}
			- (1+o(1)) c_n^2 + \log |\psi''(t_n)| \to - \infty.
		\ee 
		By Lemma~\ref{lem:psit} (iii), as $t \searrow 0$, we may estimate $|\psi''(t)|$ as follows: fix $t_0>0$ and take $t\in (0,t_0)$, then 
		\begin{align}
			\log (- \psi''(t)) & = \log (- \psi''(t_0)) - \int_t^{t_0} \frac{ \psi'''(u)}{\psi''(u)} \dd u \leq \log (- \psi''(t_0)) + \int_t^{t_0} \frac{C}{u} \dd u \notag \\ 
			& = \log (- \psi''(t_0)) + C \log \frac{t_0}{t} = O(|\log t|). 
		\end{align} \label{eq:psiseclog}
		In particular 
		 $\log |\psi''(t_n)| = O(\log t_n)$. 
		On the other hand 
		\be 
			\frac{c_n^2}{|\log t_n|} = \Bigl(\frac{c_n \eps_n}{t_n}\Bigr)^2 \frac{\Psi_n''(t_n)}{\psi''(t_n)}  \frac{t_n^2 |\psi''(t_n)|}{|\log t_n|}
		\ee 
		The ratio $\Psi''_n(t_n)/\psi''(t_n)$ stays bounded away from $0$ and by Lemma~\ref{lem:psit}, 
		\be \label{eq:psisecgrowth}
			\frac{t^2 |\psi''(t)|}{|\log t|} \to \infty.
		\ee 
		 Thus we may find a function $\omega(t) \to 0$ such that $\omega(t)^2 t^2 |\psi''(t)|/|\log t|$ still goes to infinity as $t\searrow 0$, set $c_n \eps_n = t_n \omega(t_n)$, and then Eq.~\eqref{eq:gtails} holds true. The bound~\eqref{eq:integrand-tails} then shows 
		\be \label{eq:gaussian-tails}
			\int_0^{t_n - c_n \eps_n } \sqrt{|\psi''(t)|} \e{\Psi_n(t)} \dd t 
				+ \int_{t_n + c_n \eps_n }^{t'_n} \sqrt{|\psi''(t)|} \e{\Psi_n(t)} \dd t 
				= o \Bigl( \e{\Psi_n(t_n)} \Bigr).
		\ee
		Eqs.~\eqref{eq:gaussian-dominant} and~ \eqref{eq:gaussian-tails} complete the proof of the lemma. 
\end{proof} 

\begin{lemma} \label{lem:gaussian2}
	Let $N_n\to \infty$ with $N_n^*\ll N_n \ll n$. Then 
	$$\int_0^{t'_n} \sqrt{|\psi''(t)|} \e{\Psi_n(t)} \dd t 
		\sim \sqrt{2\pi}\, \e{\Psi_n(t_n)}. $$
\end{lemma}

\begin{proof}
	The proof of Lemma~\ref{lem:gaussian} applies without any changes, the end result simplifies because $1- n \sigma^2 q''(\zeta_n) \to 1$ by Lemma~\ref{lem:hessian}. 
\end{proof} 

\begin{lemma} \label{lem:gaussian3} 
	Let $N_n\to \infty$ with $\liminf (N_n/n)>0$. Let $\delta_0>0$ as in Lemma~\ref{lem:crit-large}. Then 
	$$\int_0^{\delta_0} \sqrt{|\psi''(t)|} \e{\Psi_n(t)} \dd t 
		\sim \sqrt{2\pi}\, \e{\Psi_n(t_n)}. $$
\end{lemma}

\begin{proof}
	By Lemma~\ref{lem:crit-large}, $\Psi_n(t)$ has a unique critical point $t_n$ in $(0,\delta_0)$ and we may use a Gaussian approximation on this interval---there is no need to restrict to an interval whose length goes to zero. Apart from this difference, the proof is identical to Lemmas~\ref{lem:gaussian} and~\ref{lem:gaussian2}. 
\end{proof}

\section{Evaluation of contour integrals. Proof of the main theorems} \label{sec:asyana}

The proof of the main theorems starts from the decomposition 
\be \label{eq:decomposition}
	\P(S_n = n \mu +N_n) = H_n + V_n
\ee
where $H_n$ and $V_n$ are defined as in~\eqref{eq:hvdef}.
The correctness of~\eqref{eq:decomposition} is checked as in Section~\ref{sec:strategy}, building on the properties of $G(z)$ proven in Section~\ref{sec:lindeloef}. For the proof of Theorem~\ref{thm:critical} it is convenient to decompose $H_n$ further as $H_n = H_n^1 + H_n^2$ where 
\be \label{eq:hnone}
	H_n^1 = \frac{1}{\pi} \int_0^{t'_n} \e{n\Re \varphi(t) - (\mu n +N_n)t} \sin (n \Im \varphi(t)) \dd t
\ee
and $H_n^2$ is a similar integral, but with integration from $t'_n$ to $\eta_n$. 
In the proof of Theorem~\ref{thm:large} we adopt slightly modified definitions and replace the sequence $\eta_n \sim N_n/(n\sigma^2)$ in the domain of integration 
by another sequence $\eps_n \searrow 0$ or by some fixed small $\eps>0$. 

\subsection{Evaluation of $V_n$}

Theorem~\ref{thm:large} only needs upper bounds for $V_n$, provided in Lemmas~\ref{lem:arc} and~\ref{lem:arc2}. Theorems~\ref{thm:moderate} and~\ref{thm:critical} requires the full asymptotic behavior of $V_n$ proven in Lemma~\ref{lem:arcmod}. 

\begin{lemma} \label{lem:arc}
  Suppose $N_n \to \infty$ and $N_n \geq \delta n$ for some $\delta>0$ and all $n\in \bbN$. Then for suitable $C=C_\delta>0$, every sufficiently small $\eta>0$, and all $n\in \bbN$ 
 \begin{equation*} 
	 \frac{1}{\pi}  \int_0^{\pi} \e{n  \Re \phi(\eta + \ii \theta) - (\mu n + N_n) \eta }  \dd \theta
       \leq  \e{- C n \eta}. 
  \end{equation*} 
\end{lemma}

\begin{proof} 
   Let $\mathcal{S}_+$ be the half-strip $\{t \in \bbC\mid \Re t \geq 0,\, \Im t \in [0,\pi )\}$.   By Theorem~\ref{thm:bval}, we know that as $t\to 0$ in $\mathcal{S}_+$
   \be \label{eq:cornerbound}
   \Re \bigl(\phi(t) - \mu t - \frac{N_n}{n} t\bigr)  =  - \frac{N_n}{n} \Re t + \frac{1}{2} \sigma^2 \bigl( (\Re t)^2 - (\Im t)^2\bigr) + O(t^3).
   \ee
    Then for sufficiently small $\eps_1>0$ and all $t\in \mathcal{S}_+$ with $\max(|\Re t|,|\Im t|) \leq \eps_1$, the right side of Eq.~\eqref{eq:cornerbound} is smaller than $- \delta \Re t /2$, which shows 
    \be 
	\frac{1}{\pi}  \int_0^{\eps_1} \e{n \Re \phi(\eta + \ii \theta) - (\mu n + N_n) \eta} \dd \theta
    	       \leq \frac{\eps_1}{\pi} \e{- n \delta \eta/2} 
    \ee 
	for all $\eta \in (0,\eps_1)$. On the other hand on the unit circle $G(z)$ is given by a power series with strictly positive coefficients and therefore $|G(z)|$ has a unique maximum at $z=1$. It follows that for all $t = \ii \theta$ with $\theta \in [\eps_1, 2\pi - \eps_1]$, we know $\Re \varphi(t)<0$.  By continuity this extends to some thin vertical strip $\Im t\in [\eps_1,2\pi-\eps_1]$, $\Re t\in [0,\eps_2]$ so that $\Re (\varphi(t) - \mu t) \leq - \mu \Re t$ and 
	\be \label{eq:arc2}
		\frac{1}{\pi}  \int_{\eps_1}^\pi \e{n \Re \phi(\eta + \ii \theta) - (\mu n + N_n) \eta }  \dd \theta \leq  \frac{\pi - \eps_1}{\pi} \e{- n\mu \eta} 
	\ee 
	for all $\eta \in (0,\eps_2)$. To conclude, we let $C_\delta:= \min(\delta/2, \mu)$. 
\end{proof}

\begin{lemma} \label{lem:arc2}
  Assume $N_n = m- n \mu \to \infty$ and $N_n = o(n)$. Let $\eps_n \searrow 0$ with $\eps_n \leq \eta_n = (1+o(1))\frac{N_n}{n \sigma^2}$. Then for suitable constant $C>0$, as $n\to \infty$,
 \begin{equation*} 
  \Bigl|\frac{1}{2\pi}   \int_0^{\pi} \e{n \Re \phi(\eps_n + \ii \theta) - m \eps_n)} 
  \dd \theta\Bigr|
       \leq \e{- (1+o(1)) N_n \eps_n/2}.
  \end{equation*} 
\end{lemma}

\begin{proof} 
   The proof is analogous to Lemma~\ref{lem:arc}. We start from the estimate~\eqref{eq:cornerbound}. In a sufficiently small $\eps_1$-neighborhood of the  corner $0$ of the half-strip $\mathcal{S}_+$, we have 
   \be \label{eq:aml}
    \Re \bigl(\phi(t) - \frac{m}{n} t\bigr) \leq   - \Bigl( \frac{N_n}{n} - \frac{1}{2} \sigma^2 \Re t \Bigr)   \Re t. 
   \ee
   Notice that $\eps_1$ can be chosen $n$-independent: we only need $- \sigma^2 (\Im t)^2 + O((\Im t)^3) \leq 0$ for $|\Im t|\leq \eps_1$. 
  When $\Re t= \eps_n$ with $\eps_n \leq \frac{N_n}{ n \sigma^2}$, the upper bound in~\eqref{eq:aml} is in turn bounded by $- (1+o(1)) N_n \eps_n /2$. 
  
  When $\Re t= \eps_n$ is small but $\Im t$ is bounded away from $0$ and $2 \pi$, we estimate  
  \be \label{eq:middlebound} 
   \Re \Bigl(\phi(t)- \frac{m}{n} t\Bigr) \leq - \frac{m}{n} \Re t = - \Bigl( \mu + \frac{N_n}{n}\Bigr) t \leq - \frac{N_n}{n} t. 
  \ee
  and we conclude as in Lemma~\ref{lem:arc}.
\end{proof} 

The proof of the next lemma is closely related to the treatment of moderate deviations for random variables with generating functions analytic beyond $z=1$ given by Ibragimov and Linnik~\cite{ibragimov-linnik}.

\begin{lemma} \label{lem:arcmod}
	Let $N_n \to \infty$ along $\sqrt{n} \ll N_n = O( n^{1-\gamma})$ for some $\gamma>0$. Define $\eta_n \sim N_n/(n \sigma^2)$ by~\eqref{eq:etan} and $V_n$ as in~\eqref{eq:hvdef}. 
	Then Eq.~\eqref{eq:vnfinal} holds true. 
\end{lemma}

\begin{proof}
	Since $\Re \phi'(t) = \mu + \sigma^2 t + O(t^2)$ is strictly increasing for small $t>0$ and $N_n / n \to 0$, we may fix $\delta>0$ small enough so that 
	for large $n\geq n_\delta$, the equation~\eqref{eq:etan} has indeed a unique solution $\eta_n \in (0,\delta)$, which satisfies $\eta_n \sim N_n/(n\sigma^2)$. 
	Arguments analogous to the proof of Eq.~\eqref{eq:arc2} show 
	\be \label{eq:arc3}
		\frac{1}{\pi} \Bigl| \int_{\delta}^\pi \e{n (\varphi(\eta_n + \ii \theta)) - (\mu n + N_n) (\eta_n + \ii \theta)} \cos (n \Im \varphi(\eta_n + \ii \theta)) \dd \theta \Bigr| \leq \e{- n \mu \eta_n}.
	\ee
	We also have, uniformly in $s \in (0,\delta_n)$, 
   \begin{align}  
     \phi(\eta_n + \ii s) & = \phi(\eta_n) + \phi'(\eta_n) \ii s - \frac{1}{2} \phi''(\eta_n) s^2 + O(s^3)\notag  \\
      & = \phi(\eta_n) + \ii \bigl( \mu+ \frac{N_n}{n}\bigr) s - \Im \phi'(\eta_n) s - \frac{1}{2} \phi''(\eta_n) s^2 + O(s^3)
   \end{align} 	
	and therefore 
    \begin{multline} 
	    \Re \phi(\eta_n + \ii s) - \Bigl(\mu + \frac{N_n}{n}\Bigr) \eta_n 
	    = \Re \phi(\eta_n) - \frac{1}{2} \phi''(\eta_n) s^2 (1+O(s)) -  \Im \phi'(\eta_n) s.
    \end{multline}
    Since $\varphi''(\eta_n) \to \sigma^2$, standard arguments show 
  \be  \label{eq:arc4}
  	\frac{1}{\pi} \int_0^\delta \e{ n \Re  \varphi(\eta_n) - \frac{n}{2} \Re  \phi''(\eta_n) s^2 (1+O(s))} \dd s \sim \frac{\exp(n \Re \varphi(\eta_n))}{\sqrt{2 \pi n \sigma^2}}, 
  \ee 
  	moreover the contribution to the interval from $s \geq \delta_n:= (\log n)/\sqrt{n}$ is negligible and Eq.~\eqref{eq:arc4} holds true with $\delta$ replaced by $\delta_n$. By Theorem~\ref{thm:bval} and the relation $\Im G(\e{t}) = |G(\e{t})| \Im \varphi(t)$ the imaginary part of $\varphi'(\eta_n)$ vanish faster than any power of $\eta_n =O(n^{-\gamma})$,  hence $n \Im \varphi'(\eta_n)$ can be neglected; the same argument works for $n \Im \varphi''(\eta_n)$. For the cosine, we look separately at $(0,\delta_n)$ and $(\delta_n,\delta)$. On $(0, \delta_n)$, again by Theorem~\ref{thm:bval}, $\sup_{s \in (0,\delta_n)}|\Im \varphi(\eta_n + \ii s)|$, vanishes faster than any power of $\max (\eta_n, \delta_n)$ hence it can be neglected. On $(\delta_n,\delta)$ we simply bound the cosine by $1$. As the contribution from $(\delta_n,\delta)$ to the integral~\eqref{eq:arc4} is negligible, combining with~\eqref{eq:arc3} we find in the end 
  	\be \label{eq:vndeco}
  		V_n = (1+o(1)) \frac{ \exp(n \Re \varphi(\eta_n) - \mu n - N_n)}{\sqrt{2 \pi n \sigma^2}} + O \bigl( \e{- n \mu \eta_n} ). 
  	\ee
  	By Theorem~\ref{thm:bval} and Definition~\ref{def:cramer}, we have 
  	\be 
		n\Re \varphi(\eta_n) - \mu n - N_n =\Bigl(1+O(\frac{N_n}{n}) \Bigr) \frac{N_n^2}{2 n \sigma^2} 
	\ee
	 with correction terms expressed in terms of the Cram{\'e}r series. In particular, the exponent goes to $- \infty$ as $ -N_n^2/ n$, i.e., slower than the term $- n \mu \eta_n = - N_n \mu$ in the second term. Therefore the second term in Eq.~\eqref{eq:vndeco} is negligible compared to the first and Eqs.~\eqref{eq:veval} and~\eqref{eq:vnfinal} hold true. 
\end{proof}

\subsection{Evaluation of $H_n$}
Here we focus on the regime $\liminf N_n/N_n^*>1$; the case $\limsup N_n/N_n^* \leq 1$ is treated in the proof of Theorem~\ref{thm:moderate}. In order to apply the Gaussian approximation from Section~\ref{sec:gaussian}, we need to drop the sine and replace $\Im \varphi(t)$ with $G(\e{t})$. 

\begin{lemma} \label{lem:sineaway}
  Let $\eps_n \searrow 0$ faster than some power of $n$, i.e.,  $n \eps_n^p \to 0$ for some $p>0$. Then 
  \begin{equation*}
     \sin \bigl( n \Im \phi(t) \bigr) \sim n \Im G(\e{t})
   \end{equation*}
   uniformly for $t \in [0,\eps_n]$. 
\end{lemma}

\begin{proof} 
 We have 
 \be \label{eq:gphi}
 \Im G(\e{t}) = \Im \e{\phi(t)} = \e{\Re \phi(t)} \Im \phi(t) = |G(\e{t})|\, \Im \phi(t). 
  \ee
  We know that 
  \be 
     \sup_{t\in [0,\eps_n]}|\Re G(\e{t}) -1| = O(\eps_n) \to 0,
  \ee
  and the imaginary part vanishes faster than any power, in particular  
  \be
     \sup_{t\in [0,\eps_n]} |\Im G(\e{t})| = O( \eps_n^p) \to 0. 
  \ee
  It follows that a similar bound applies to $\Im \phi(t)$. As a consequence 
  \begin{multline}
    \sup_{t\in [0,\eps_n]} \Bigl|\frac{ n \Im \phi(t) - \sin (n \Im \phi(t))}{\sin(n \Im \phi(t))} \Bigr| = O \Bigl( n^2 \sup_{t \in [0,\eps_n]} |\Im \phi(t)|^2\Bigr)\\
    = O\Bigl( n^2 \eps_n^{2p}\Bigr) \to 0, 
  \end{multline}
  Thus $\sin (n \Im \phi(t)) \sim n \Im \phi(t)$, uniformly in $[0,\eps_n]$. Eq.~\eqref{eq:gphi} in turn shows $\Im \phi(t) \sim \Im G(\e{t})$ uniformly in $[0,\eps_n]$.  
\end{proof}

The dominant contribution to the Gaussian integral in Lemmas~\ref{lem:gaussian}-\ref{lem:gaussian3} comes from windows of width $o(t_n)$ around $t_n$; by Remark~\ref{rem:tn} and Lemma~\ref{lem:critsequences}, $t_n = O(t_n^*) = O(N_n^*/n) = O(n^{- [1-\alpha]/[2-\alpha]})$. This latter bound vanishes like some negative power of $n$, hence Lemma~\ref{lem:sineaway} is applicable on the interval contributing most to the Gaussian integrals. Outside we use the inequality $|\sin (n \Im \varphi(t))|\leq n \Im \varphi(t) = n (1+o(1)) \Im G(\e{t})$, and we find: For $N_n \to \infty$ with $\liminf_{n \to \infty} N_n/N_n^*>1$, we have 
\be \label{eq:g1}
	H_n^1\sim \frac{n}{\sqrt{1- n \sigma^2 q''(\zeta_n)}}\, \e{\Psi_n(t_n)}.
\ee
For $N_n \to \infty$ with $N_n \gg N_n^*$, 
\be \label{eq:g2}
	H_n^1 \sim n \e{\Psi_n(t_n)}.
\ee
$H_n^2$ is estimated in the proof of Theorem~\ref{thm:critical} and is not needed in the proof of Theorem~\ref{thm:large}. Finally for $\liminf N_n/n >0$ and $\delta_0>0$ small enough as in Lemma~\ref{lem:crit-large}, 
\be \label{eq:g3}
	\frac{1}{\pi} \int_0^{\delta_0} \e{ n\Re \varphi(t) - mt} \sin \bigl( n \Im \phi(t) \bigr) \dd t \sim n \e{\Psi_n(t_n)}. 
\ee
Remember that $\Psi_n(t_n) = \Phi_n(t_n,\zeta_n) = - f_{nr}(x_{nr}) + o(1)$ and $1-n \sigma^2 |q''(\zeta_n)| \sim 1- n \sigma^2 |q''(x_{nr})|$ by the definition of $\Psi_n$ and Eq.~\eqref{eq:phitrunc}, so the right-hand sides in Eqs.~\eqref{eq:g1}--\eqref{eq:g3} correspond to the relevant contribution in Theorems~\ref{thm:critical} and~\ref{thm:large}. 

\subsection{Critical scale: proof of Theorem~\ref{thm:critical}}

Let $N_n \to \infty$ with $\liminf N_n/N_n^* >1$ and $N_n = O(n^{1/(2-\alpha)})$. 
Let $\eta_n \sim N_n/(n\sigma^2)$ be the solution of~\eqref{eq:etan},  define $H_n$ and $V_n$ as in~\eqref{eq:hvdef}, and $H_n^1$ and $H_n^2$ as in~\eqref{eq:hnone}
Thus we have 
\be 
	\P(S_n = n \mu + N_n) = V_n + H_n^1 + H_n^2.
\ee
 By Lemma~\ref{lem:arcmod}, the asymptotics of $V_n$ is given in terms of the  Cram{\'e}r series as in Eq.~\eqref{eq:vnfinal}. $H_n^1$ is evaluated with Lemma~\ref{lem:gaussian} and Lemma~\ref{lem:gaussian2} as the right-hand side of Eq.~\eqref{eq:hnfinal}. The proof is complete once we show $H_n^2 = o(V_n)$. 

On $(t'_n,\eta_n)$ the function $\Psi_n(t)$ is increasing by Lemma~\ref{lem:crit-medium}, 
\be \label{eq:critexpob}
	\sup_{t \in (t'_n, \eta_n)}	\Psi_n(t)  \leq \Psi_n(\eta_n) = \Bigl( n \Re \varphi(\eta_n) - (n \mu+ N_n) \eta_n \Bigr) + \psi(\eta_n). 
\ee
The term in big parentheses can be reexpressed with the Cram{\'e}r series and is exactly equal to the exponent in the evaluation of $V_n$ (see Eqs.~\eqref{eq:veval} and~\eqref{eq:vnfinal}), while $\psi(\eta_n) \to - \infty$. The prefactor satisfies
\be 
	\sup_{t \in (t'_n, \eta_n)} \sqrt{|\psi''(t)|} \leq \sqrt{|\psi''(t'_n)|} = \exp \Bigl( O(\log t'_n) \Bigr) = \exp \Bigl(O(\log n)\Bigr). 
\ee
(remember~\eqref{eq:psiseclog} and $t'_n \geq t^*_n$, with $t^*_n$ of the order of $N_n^*/n \gg 1/\sqrt{n})$. On the other hand $\eta_n \searrow 0$ faster than some power of $n$, so by Lemma~\ref{lem:psit} we have for some constant $C$
\be \label{eq:psietan}
	|\psi(\eta_n) |\gg |\log \eta_n|\geq C \log n \gg \log \sqrt{|\psi''(t'_n)|},
\ee 
whence 
\be 
	\psi(\eta_n) +   \log n + \sup_{t \in (t'_n,\eta_n)} \log \sqrt{|\psi''(t)|} \to - \infty. 
\ee 
In view of~\eqref{eq:critexpob},~\eqref{eq:veval} and~\eqref{eq:vnfinal}, we obtain
\be \label{eq:crithnrem}
	H_n^2 = \frac{n}{\pi} \int_{t'_n}^{\eta_n} \e{n \Re \varphi(t) -( n \mu + N_n) t} \Im G(\e{t}) \dd t = o(V_n).
\ee
which concludes the proof. \hfill $\qed$

\subsection{Big jump: proof of Theorem~\ref{thm:large}}

Let $N_n \to \infty$ with $N_n \gg n^{1/(2-\alpha)}$. Notice that, by Lemma~\ref{lem:critsequences}, we then have $N_n \gg N_n^{**}$. We distinguish two cases. 

\emph{Case 1: $\liminf_{n\to \infty} N_n/n>0$.}
Fix $\delta_0>0$ as in Lemma~\ref{lem:crit-large} and define $H_n$ and $V_n$ as in~\eqref{eq:hvdef} but with $\delta_0$ instead of $\eta_n$. 
The asymptotic behavior of $H_n$ is given by Eq.~\eqref{eq:g3}. The proof is complete once we check $V_n = o(H_n)$. 

 $V_n$ is exponentially small in $n$ by Lemma~\ref{lem:arc}. 
Remembering Eqs.~\eqref{eq:phitrunc} and~\eqref{eq:fnrq} we get
\begin{align} 
	\Psi_n(t_n) & = - f_{nr}(x_{nr}) +o(1) = - q(N_n) + \frac{1}{2} (1+o(1)) n \sigma^2 q'(N_n)^2 +o(1) \notag  \\
			& \geq - q(N_n) +o(1) \geq - C N_n^{\alpha}. \label{eq:grolo}
\end{align}
It follows that $H_n \sim n \exp(\Psi_n(t_n))$ goes to zero slower than $\exp(- c n^\alpha)$ for some $c>0$, hence $V_n = o(H_n)$ and 
\be \label{eq:grolol}
	\P(S_n = n \mu + N_n) \sim H_n \sim n \e{- f_{nr}(x_{nr})}.
\ee
If $n^{1/(2-\alpha)} \ll N_n \ll n$, set $\eps_n = t'_n$. By Remark~\ref{rem:tn}, the critical point $t'_n \in (t_n^*,\eta_n)$ is bounded from below by some constant times $\eta_n \sim N_n/(n\sigma^2)$. We define $H_n$ and $V_n$ as in~\eqref{eq:hvdef} but with $\eps_n$ instead of $\eta_n$. $H_n$ is evaluated as in~\eqref{eq:g2}, which yields $H_n \sim n \exp(\Psi_n(t_n))$; Eq.~\eqref{eq:grolo} stays valid. $V_n$ is estimated by Lemma~\ref{lem:arc2}, which yields $V_n = O(\exp(- N_n t'_n/2)))$. 
Now $N_n t'_n \to \infty$ much faster than $N_n^{\alpha}$. Indeed $t'_n$ is bounded from below by some constant times $N_n/n$, hence 
\be
	\frac{N_n^\alpha}{N_n t'_n} = O(n N_n^{\alpha-2}) \to 0. 
\ee
It follows that $V_n = o(H_n)$ and Eq.~\eqref{eq:grolol} stays true. \hfill $\qed$ 
 
\subsection{Small steps: proof of Theorem~\ref{thm:moderate}} 

The proof of Theorem~\ref{thm:moderate} requires two more technical lemmas, proven at the end of this section. Remember the function $R_n(t)$ from~\eqref{eq:rndef}. 

\begin{lemma} \label{lem:sucri}
	If $\limsup_{n \to \infty} N_n/N_n^* <1$, then $R'_n >0$ on $(0,\eta_n)$ for all sufficiently large $n$. 
\end{lemma}

\begin{lemma} \label{lem:rncrit}
	Let $N_n \to \infty$ with $N_n \sim N_n^*$. Set $f_n^* (x):= q(x) + \frac{(N_n^* - x)^2}{2 n \sigma^2}$. Suppose that there are infinitely many $n\in \N$ for which the equation $R'_n(t) =0$ has a solution  $s_n \in(0,\eta_n)$.  Then $s_n\sim  t^*_n $ and 
	$$R_n(s_n) = - f_n^*(x_n^*) +o \Bigl( \frac{(N_n^*)^2)}{n}\Bigr) + O(\log n). $$	
\end{lemma}

\noindent The zero $s_n$ need not be unique---in case of non-uniqueness the lemma applies to every choice of $s_n$. 

\begin{proof} [Proof of Theorem~\ref{thm:moderate}]
Let $N_n \to \infty$ along $\sqrt{n} \ll N_n \leq (1+o(1)) N_n^*$. 
Let $\eta_n \sim N_n/(n\sigma^2)$ be the solution of~\eqref{eq:etan} and define $V_n$ and $H_n$ as in~\eqref{eq:hvdef}. By Lemma~\ref{lem:arcmod}, the asymptotic behavior of $V_n$ is given by Eqs.~\eqref{eq:veval} and~\ref{eq:vnfinal}, so it remains to verify that $H_n = o(V_n)$. We estimate
\be 
	H_n \leq \frac{n \eta_n}{\pi} \sup_{t\in (0,\eta_n)} \sqrt{|\psi''(t)} \e{\Psi_n(t)}  = \frac{n \eta_n}{\pi} \sup_{t\in (0,\eta_n)} \e{R_n(t)}. 
\ee
Just as in Lemma~\ref{lem:prefactor}, one checks that $R_n(t)$ is increasing for small $t$. We distinguish two cases. 

\emph{Case 1:} $\liminf_{n\to \infty} N_n/N_n^*<1$. Then Lemma~\ref{lem:sucri} shows, for large $n$, 
\be 
	\sup_{t\in (0,\eta_n)}R_n(t) \leq R_n(\eta_n)
		 = \Bigl( n \Re \varphi(\eta_n) - (n \mu +N_n)\eta_n\Bigr) + \psi(\eta_n) + \log \sqrt{|\psi''(\eta_n)|}
\ee
and we deduce from~\eqref{eq:veval} 
\begin{align} 
	\frac{H_n}{V_n} & \leq \frac{n \eta_n}{\pi} \sqrt{2\pi n \sigma^2} \exp\Bigl(\psi(\eta_n) + \log \sqrt{|\psi''(\eta_n)|}\Bigr) \notag \\
		& = \exp \Bigl(\psi(\eta_n) + \log \sqrt{|\psi''(\eta_n)|} + O(\log n) \Bigr) \label{eq:modcase1}
\end{align}
which goes to zero by an estimate analogous to~\eqref{eq:psietan}.

\emph{Case 2:} $N_n \sim N_n^*$. If $R'_n$ reaches its maximum at $t=\eta_n$, the estimate~\eqref{eq:modcase1} still applies. If along some subsequence $(n_j)$,  $R'_n$ reaches its maximum at some interior point $s_n \in (0,\eta_n)$, then we must have $R'_n(s_n) =0$ and by Lemmas~\ref{lem:rncrit} and ~\ref{lem:fncri}
\be 
	\sup_{t\in (0,\eta_n)} R_n(t) \leq - (1+\eps + o(1)) \frac{{N_n^*}^2}{2n \sigma^2} + O(\log n). 
\ee
for some $\eps>0$ and all large $n$. Since $V_n = \exp( - (1+o(1)) \frac{{N_n^*}^2}{2n \sigma^2} + O(\log n) )$ we get 
\be \label{eq:modcase2}
	\frac{H_n}{V_n} \leq \exp \Bigl( - (\eps+o(1)) \frac{{N_n^*}^2}{2n \sigma^2} + O(\log n) \Bigr). 
\ee
By Assumption~\ref{ass:convexity} and Lemma~\ref{lem:critsequences},
\be \label{eq:lastesti}
	\frac{1}{n\sigma^2} = |q''(x_n^*)| \gg \frac{\log x_n^*}{{x_n^*}^2} \geq \frac{\log n + O(1)}{{x_n^*}^2}
\ee
so $\log n=o({N_n^*}^2/n)$ and the right-hand of~\eqref{eq:modcase2} goes to zero. 

We have checked in both cases that $H_n = o(V_n)$, which concludes the proof. 
\end{proof}

\begin{proof}[Proof of Lemma~\ref{lem:sucri}]
	Suppose that the equation $R'_n(s_n) =0$ has a solution $s_n \in (0,\eta_n)$ for infinitely many $n$. Thus $(s_n)$ may be defined only along some subsequence $(n_j)$, which we suppress from the notation. 
	Define $y_n \sim \psi'(s_n)$ as in~\eqref{eq:yn}. Proceeding as in Lemma~\ref{lem:prefactor}, we find that $q'(y_n) \sim (N_n - y_n)/(n \sigma^2)$. By the convexity of $q'$ and the definition of $N_n^*$, we have 
	\be 
		(1+o(1)) \frac{N_n - y_n}{n \sigma^2} = q'(y_n) \geq \frac{N_n^* - y_n}{n\sigma^2 } = \frac{N_n^* - N_n}{n\sigma^2 } + \frac{N_n - y_n}{n\sigma^2}
	\ee	
	hence $N_n^* - N_n \leq o (N_n-y_n) = o(N_n)$ and $\limsup N_n^*/N_n \leq 1$ i.e. $\liminf N_n/N_n^*>1$. 
	
	So if $\limsup_{n\to \infty} N_n/N_n^*<1$, we must have $R'_n \neq 0$ on $(0,\eta_n)$ except possibly for finitely many $n$. From the proof of Lemma~\ref{lem:prefactor} we know that $\lim_{t\searrow 0} R'_n(t) = \infty$ for all $n\in \N$, and Lemma~\ref{lem:sucri} follows. 
\end{proof}

\begin{proof} [Proof of Lemma~\ref{lem:rncrit}]
	For $t \in (0,\eta_n)$ with $\eta_n \sim N_n/(n\sigma^2)$ we have 
	\begin{align}
		R'_n(t) &=  n ( \Re \varphi'(t) - \mu) - N_n + \psi'(t) + \frac{1}{2} \frac{\psi'''(t)}{\psi''(t)} \notag \\
				 & = n \sigma^2 t - (N_n^* - \zeta(t)) + o(n t) + (N_n^*- N_n) +  O\Bigl(\frac{1}{t}\Bigr) \notag \\
				 & = n \sigma^2 t - (N_n^* - \zeta(t)) + o(N^*_n) + o(\zeta(t)).  
	\end{align}
	In terms of the dual variable $\zeta = \zeta(t)=\psi'(t)$, the equation $R'_n(t) =0$ reads 
	\be \label{eq:dualzero}
		n \sigma ^2 q'(\zeta) - (N_n^*- \zeta) = o(N_n^*) + o(\zeta). 
	\ee
	At $\zeta = x_n^*$, the left-hand side of~\eqref{eq:dualzero} vanishes by definition of $N_n^*$. It follows that 
	\be \label{eq:qcri}
		n \sigma ^2 q'(\zeta) - (N_n^*- \zeta) = \int_{x_n^*}^ \zeta (1 + n \sigma^2 q''(x)) dx. 
	\ee
	Fix $\gamma >0$. On $((1+\gamma)x_n^*,\infty)$, we have $1- n \sigma^2 |q''(x)| \geq 1 - n \sigma^2 |q''((1+\gamma)x_n^*)| =: c_\gamma >0$ and for all $x\geq x_n^*$, the integrand is non-negative. As a consequence, 
	\be
		n \sigma ^2 q'(\zeta) - (N_n^*- \zeta) \geq c_\gamma \bigl( \zeta - (1+\gamma)x_n^* \bigr)
	\ee
	for all $\zeta \geq (1+\gamma) x_n^*$. Suppose that Eq.~\eqref{eq:dualzero} has a solution $\zeta'_n$ with $\zeta'_n \geq (1+2\gamma) x_n^*$ along some subsequence. Then $N_n^* = O(x_n^*)= o(\zeta'_n)$ and 
	\be 
		c_\gamma\, \gamma \zeta'_n \leq c_\gamma ( \zeta'_n - (1+\gamma)x_n^*) = o(\zeta'_n),
	\ee
	which is a contradiction (remember $\zeta'_n \geq x_n^* \to \infty$). It follows that for every $\gamma>0$, there are at most finitely many $n$ for which $\zeta'_n >(1+2 \gamma)x_n^*$, hence $\limsup \zeta'_n / x_n^* \leq 1$. 
	
	The case $\psi'(\eta_n)\leq \zeta'_n \leq (1-2\gamma)x_n^*$ is treated in an analogous fashion, based on two observations: first, $n \sigma^2 q''(x) +1 \leq - c_\gamma <0$ for all $x\in (a,(1-\gamma) x_n^*)$ and some $c_\gamma>0$. Second, since $\eta_n$ is of the order of $N_n^*/n$ i.e., of the order of $t_n^*\sim q'(x_n^*)$, the estimate~\eqref{eq:qprimevar} shows that $\zeta'_n$ is bounded from below by some constant times $N_n^*$, i.e., we still have $N_n^*= O(\zeta'_n)$ and $\zeta'_n \to \infty$. We find $\liminf \zeta'_n /x_n^*\geq 1$, hence altogether $\lim \zeta'_n / x_n^*= 1$. This applies in particular to $\zeta'_n := \psi'(s_n)$ i.e. $s_n = q'(\zeta'_n)$.
	 Eq.~\ref{eq:qprimevar} and $t_n^* \sim q'(x_n^*)$ (Lemma~\ref{lem:inflection}) yield $\lim s_n/t_n^* = 1$. 
	 
	 By Remark~\ref{rem:tn} and Lemma~\ref{lem:critsequences} we have $s_n \sim t_n^*\geq \const N_n^*/n\gg 1/\sqrt{n}$ hence $\log s_n = O(\log n)$ and by~\eqref{eq:psiseclog}
	 \be \label{eq:c1}
	 	\log |\psi''(s_n)| = O(\log s_n) = O(\log n).
	 \ee
	 Furthermore for $t\in (0,\eta_n)$
	 \begin{align} 
	 	\Psi_n(t) & = \frac{1}{2} n \bigl( \sigma^2 t^2 + o(\eta_n^2)\bigr)  - N_n t +  t \zeta(t) - q(\zeta(t) \notag \\
	 		& = \Bigl \lbrace - q(\zeta(t)) +  \frac{1}{2} n \sigma^2 t^2 - (N_n^* - \zeta(t)) t \Bigr \rbrace + O\bigl((N_n^*- N_n)\eta_n\bigr) + o( n \eta_n^2) \label{eq:c2}
	 \end{align}	 
	 The two remainders  are $o((N_n^*)^2/n)$ by our choice of $N_n$. Write $g_n(t)$ for the term in curly braces. At $t = q'(x_n^*) = [N_n^*- x_n^*]^2/[n \sigma^2]$ we have  $\zeta(t) = x_n^*$, $g_n(t) = - f_n^*(x_n^*)$ and $g'_n(t) = 0$. Moreover for $t \in (0,\eta_n)$ 
	 \be \label{eq:gns}
	 	g_n''(t) = n \sigma^2 + \zeta'(t) = n \sigma^2 - \frac{1}{|q''(\zeta(t))|} = O(n \sigma^2). 
	 \ee
	 Here we have used that $t\leq \eta_n$ implies that $\zeta(t)$ is bounded from below by a constant times $N_n^*$ or equivalently, $x_n^*$ and therefore $|q''(\zeta(t))|$ is bounded from below by some constant times $1/(n \sigma^2)$. We deduce 
	 \be \label{eq:c3}
	 	\bigl|g_n (s_n) + f_n^*(x_n^*)\bigr| \leq \frac{1}{2} (t_n - q'(x_n^*))^2 \sup_{0,\eta_n} |g''_n| = o(t_n^2 n \sigma^2)  = o(n \eta_n^2) = o\Bigl(\frac{(N_n^*)^2}{n}\Bigr). 
	 \ee
	The estimate on $R_n(s_n)$ follows from Eqs.~\eqref{eq:c1}, \eqref{eq:c2} and~\eqref{eq:c3}.	 
\end{proof}

\appendix

\section{Proofs of Lemmas~\ref{lem:stretched} and~\ref{lem:oklog}} \label{app:scales}

\begin{proof}[Proof of Lemma~\ref{lem:stretched}]
	The proof of Assumption~\ref{ass:convexity} is straightforward and left to the reader. The function $q(\zeta) =  \zeta^\alpha$ is analytic in $\Re \zeta >0$
	and $p(\zeta)= \exp (- \zeta^\alpha)$ satisfies, for all $k \in \bbN$, 
  \be \label{eq:xikb}
     |\zeta^k p(\zeta)| = |\zeta|^k \e{- |\zeta|^\alpha \cos(\alpha {\rm arg}(\zeta))} 
      \leq |\zeta|^k \e{- |\zeta|^\alpha \cos(\alpha \pi /2)}.
  \ee
  Since $\alpha \in (0,1)$, we have $\cos \frac{\alpha \pi}{2} >0$ and Eq.~\eqref{eq:xikb} shows that $|\zeta^k p(\zeta)|$ is integrable along $\Re \zeta =1/2$ and that $p(\zeta)$ grows slower than any exponential $\exp(\eps |\zeta|)$. This proves Assumption~\ref{ass:analyticity}. 
  
  The equation $q''(x_n^*) = -1/(n\sigma^2)$ can be solved explicitly. $N_n^*$ and $N_n^{**}$ are best determined with the scaling relation~\eqref{eq:stretched-scaling}. They have already been determined in~\cite{nagaev}, we omit the proof. For the insensitivity scale, we notice that 
  \be 
  	n \sigma^2 q'(N_n)^2 = n \sigma^2 \alpha^2 N_n^{2 \alpha-2}
  \ee
  which goes to zero if and only if $N_n \gg n^{-1/(2-2\alpha)}$. 
\end{proof} 

\begin{proof} [Proof of Lemma~\ref{lem:oklog}]
	The function $q(x) = - \log c +(\log x)^\beta$ is clearly smooth on $(1,\infty)$. Then $q'(x) = \beta  (\log x)^{\beta -1}/ x$ and as $x\to \infty$, 
	\be \label{eq:oklog2}
		q''(x) \sim -  \frac{\beta  (\log x)^{\beta -1}}{x^2},\quad q'''(x) \sim \frac{2 \beta  (\log x)^{\beta -1}}{x^3}.
	\ee
	Assumption~\ref{ass:convexity} is easily checked. For Assumption~\ref{ass:analyticity} we note $q(\zeta) = c + (\log \zeta)^\beta$ is analytic in $\Re \zeta >1$. Fix $b>1$ and write $\zeta = r \exp( \ii \theta)$. 
	As $|\zeta| \to \infty$ along $\Re \zeta = b$ i.e. $\zeta = b + \ii y$, the argument $\theta$ goes to $\pm \pi/2$ and we have 
	\begin{align}
		\Re (\log |\zeta| + \ii \theta )^\beta & = \Re \Bigl( \log |y| + \frac{1}{2}\log \Bigl(1+ \frac{b^2}{y^2}\Bigr)  + \ii \theta\Bigr)^\beta \\
		& = (\log |y|)^\beta + o(1),
	\end{align}
	conditions(i) and~(ii) in Assumption~\ref{ass:analyticity} are easily checked. Assumption~\ref{ass:analyticity2}(iii) follows from a computation similar to~\eqref{eq:oklog2}. Set $y_0 = \sqrt{r^2 - b^2}$. We have for $\zeta = b + \ii y$, $y\geq y_r$, uniformly in $r$,
	\begin{align}
		\Re q(\zeta) - \Re q(z_r) & = (\log y)^\beta - (\log y_0)^\beta + o(1)\notag \\
			& \geq \beta (\log y_0)^{\beta-1} \log \frac{y}{y_0}  + o(1)
	\end{align}
	hence
	\begin{align} 
		\int_{y_r}^\infty \e{- \Re q(b + \ii y)} \dd y 
		& \leq \e{-\Re q(b+\ii y_0) + o(1)} \int_1^\infty \e{- (\log y_0)^{\beta-1} \log s} y_0 \dd s \notag \\
		& \sim \e{-\Re q(b+\ii y_0)} \frac{y_0}{(\log y_0)^{\beta-1}}= \e{-\Re q(b+\ii y_0) + O(\log r)},
	\end{align}
	which proves Assumption\ref{ass:analyticity2}(i). Next let $\zeta \in \C$ with $\Re \zeta >1$, write $\zeta = r \exp(\ii \theta)$, then 
	\be 
		\zeta q'(\zeta) = \beta (\log \zeta)^{\beta-1} = \beta (\log r + \ii \theta)^{\beta-1} = q'(r) \Bigl(1 + \frac{\ii \theta}{\log r }\Bigr)^{\beta-1}
	\ee
	and for large $r$ and $\theta \in (0,\pi/2)$, 
	\be 
		\frac{\Im \zeta q'(\zeta)}{\Im \zeta q'(r)} \sim  \frac{\beta-1}{r \log r} \frac{\theta}{\sin \theta} \to 0 
	\ee
	and so Assumption~\ref{ass:analyticity2}(ii) holds. 
	
	We now turn to the asymptotic behavior of the sequences $x_n^*$, $N_n^*$ and $N_n^{**}$. Since $q''(x_n) = -\frac1{n\sigma^2}$, it is clear that $x_n \to \infty$ as $n \to \infty$. The equation is
\be
\label{wie geht's?}
\frac1{n\sigma^2} \sim \frac{\beta (\log x_n^*)^{\beta-1}}{(x_n^*)^2}.
\ee
Consequently,
\be
(x_n^*)^2 \sim \beta n \sigma^2 \bigl( \tfrac12 \log (x_n^*)^2 \bigr)^{\beta-1} 
\sim \beta n \sigma^2 (\tfrac12 \log n)^{\beta-1} \Bigl( 1 + \frac{\log \beta \sigma^2 + (\beta-1) \log \log x_n^*}{\log n} \Bigr)^{\beta-1}.
\ee
The last bracket is asymptotically equal to 1 and we get the expression for $x_n^*$. Next, we have from \eqref{eq:xnstar}
\be
N_n^* = x_n^* + n \sigma^2 \frac{\beta (\log x_n^*)^{\beta-1}}{x_n^*} \sim 2 x_n^*.
\ee
The last asymptotics follows from \eqref{wie geht's?}.

We now turn to $N_n^{**}$. It is asymptotically given by the solution of the equations
\begin{align}
&\frac{N_n^2}{2n\sigma^2} = q(x_n) + \frac{(N_n-x_n)^2}{2n\sigma^2}, \label{Satz eins} \\
&q'(x_n) = \frac{N_n-x_n}{n\sigma^2}. \label{Satz zwei}
\end{align}
Eq.\ \eqref{Satz zwei} is equivalent to
\be
x_n^2 - N_n x_n + \beta n \sigma^2 (\log x_n)^{\beta-1} = 0.
\ee
The relevant solution is
\be
x_n = \tfrac12 \Bigl( N_n + \sqrt{ N_n^2 - 4 \beta n \sigma^2 (\log x_n)^{\beta-1}} \Bigr) 
= N_n \Bigl( 1 - \frac{\beta n \sigma^2}{N_n^2} (\log x_n)^{\beta-1} (1 + o(1)) \Bigr).
\ee
It follows that $N_n - x_n \sim \frac{\beta n \sigma^2}{N_n} (\log x_n)^{\beta-1}$. We insert this in \eqref{Satz eins}; using $\log x_n \sim \log N_n$, we get
\be
\frac{N_n^4}{2n\sigma^2} - N_n^2 (\log N_n)^\beta - \tfrac12 \beta^2 n \sigma^2 (\log N_n)^{2\beta-2} = o(1).
\ee
The relevant solution is
\be
\begin{split}
N_n^2 &\sim n \sigma^2 \Bigl[ (\log N_n)^\beta + \sqrt{(\log N_n)^{2\beta} + \beta^2 (\log N_n)^{2\beta-2}} \Bigr] \\
&\sim 2n \sigma^2 (\tfrac12 \log N_n^2)^\beta 
\sim 2^{1-\beta} n \sigma^2 \bigl( \log n + \log 2\sigma^2 + \beta \log \log N_n \bigr)^\beta.
\end{split}
\ee
Only the term $\log n$ matters in the last bracket and the result follows.

The last part of the lemma on insensitivity sequence is shown as in~\cite[Section 8.3]{denisov-dieker-shneer} the proof is therefore omitted. 
\end{proof}

\section{Bivariate Hessian} 

As explained in Step 4 of the proof outline,  the Hessian at $(t_n,\zeta_n)$ has determinant $-1 + o(1)$ and is a saddle point of $\Phi_n(t,\zeta)$, considered as a function of two real variables $t,\zeta >0$. In order to get rid of off-diagonal elements in the Hessian and to give all eigenvalues the same sign, we take complex $\zeta$ and reparametrize, as sketched in Step 5. 

\begin{lemma} \label{lem:Fn}
   Let $\zeta(t)$ be the unique solution of $(\partial_\zeta \Phi_n)(t,\zeta) =t - q'(\zeta) = 0$.
Set  
    \begin{equation*} 
      F_n: (0,\infty) \times \bbR \to \bbC,\quad F_n(t,s) = \Phi_n (t, \zeta(t) + \ii s).
     \end{equation*} 
  Then $(\nabla F_n)(t_n,0) =0$, 
   \begin{equation*}  
       {\rm Hess}\, F_n(t,0) =\begin{pmatrix} \beta(t) & 0 \\ 0 &   q''(\zeta(t)) \end{pmatrix},  \quad \beta(t)  = - \frac{\det ({\rm Hess}\, \Phi_n)(t,\zeta(t))}{q''(\zeta(t))}.
   \end{equation*} 
\end{lemma}
\noindent Note that $\zeta_n = \zeta(t_n)$, so as $n\to \infty$ 
\be \label{eq:Fnhess}
  \det {\rm Hess}\, F_n(t_n,0) =  \beta(t_n) q''(\zeta_n)=- \det ({\rm Hess}\, \Phi_n)(t_n,\zeta_n) =  1+o(1).
\ee

\begin{proof} 
   We have 
   \begin{align}
    \partial_t F_n(t,s) & = (\partial_t \Phi_n)(t,\zeta(t) + \ii s) + (\partial_\zeta \Phi_n)(t,\zeta(t) + \ii s) \zeta'(t),\notag \\
    \partial_s F_n(t,s) & = \ii \partial_\zeta \Phi_n(t,\zeta(t) + \ii s). 
   \end{align}
   At $t=t_n$, $s=0$, we have $\zeta(t) = \zeta_n$ and $(\nabla F_n)(t_n,0) = \nabla \Phi_n(t_n,\zeta_n) =0$. For the Hessian, we compute 
   \begin{align}  
      \partial_s^2 F_n(t,0) &= - \partial_\zeta^2 \Phi_n(t, \zeta(t))= q''(\zeta(t)), \notag \\ 
      \partial_t \partial_s F_n(t,0) & = {\rm i} (\partial_t \partial_\zeta \Phi_n)(t,\zeta(t) ) + {\rm i} \partial_\zeta^2 \Phi_n(t,\zeta(t)) \notag \\
      \partial_t^2 F_n(t,0) & = \frac{\dd^2}{\dd t^2} \Phi_n(t,\zeta(t)) = \beta(t).
\end{align}
 By definition of $\zeta(t)$, 
  \be \label{eq:h1}
    0= \frac{\dd}{\dd t} \partial_\zeta \Phi_n (t,\zeta(t))   
       =  (\partial_t\partial_\zeta \Phi_n) (t,\zeta(t)) +  \partial_\zeta^2 \Phi_n(t,\zeta(t)) 
\zeta'(t). 
  \ee
  It follows that $\partial_t \partial_s F_n(t,0) =0$, 
  and 
  \begin{align}
 \beta(t) & =     \frac{\dd^2}{\dd t^2} \Phi_n(t,\zeta(t))  = \frac{\dd}{\dd t} (\partial_t \Phi_n)(t,\zeta(t)) \notag \\
      & = (\partial_t^2 \Phi_n)(t,\zeta(t)) + 
   (\partial_t \partial_\zeta \Phi_n) (t,\zeta(t)) \zeta'(t). \label{eq:h2}
  \end{align}
 We solve for $\zeta'(t)$ in Eq.~\eqref{eq:h1}, insert into Eq.~\eqref{eq:h2}, and obtain the formula for $\beta(t)$. 
\end{proof}

\noindent {\small {\bf Acknowledgement.}  The authors wish to thank the Laboratoire Jean Dieudonn\'e of the University of Nice and the Institut Henri Poincar\'e (during the program of the spring 2013 organized by M. Esteban and M. Lewin) for the kind hospitality and for the opportunity to discuss this project. S. J. thanks V. Wachtel for pointing out S.~V.~Nagaev's article~\cite{nagaev73}.}


\end{document}